\documentclass[10pt]{article}
\usepackage{epsf}
\usepackage{amsmath}

\allowdisplaybreaks

\usepackage[showframe=false]{geometry}
\usepackage{changepage}

\usepackage{epsfig}
\usepackage{amssymb}
\usepackage{amsthm}
\usepackage{setspace}
\usepackage{cite}

\usepackage{algorithmic}  % This package provides an algorithmic environment fo describing algorithms.
\usepackage{algorithm}

\usepackage{shadow}
\usepackage{fancybox}
\usepackage{fancyhdr}

\usepackage{color}
\usepackage[usenames,dvipsnames,svgnames,table]{xcolor}
\newcommand{\bl}[1]{\textcolor{blue}{#1}}

\definecolor{mypurple}{rgb}{.4,.0,.5}
\newcommand{\prp}[1]{\textcolor{mypurple}{#1}}

\usepackage[hyphens]{url}

\usepackage[colorlinks=true,
            linkcolor=black,
            urlcolor=blue,
            citecolor=purple]{hyperref}

\usepackage{breakurl}
\def\w{{\bf w}}

\def\y{{\bf y}}
\def\v{{\bf v}}
\def\x{{\bf x}}

\def\x{{\mathbf x}}

\def\w{{\bf w}}

\def\v{{\bf v}}
\def\x{{\bf x}}
\def\y{{\bf y}}
\def\z{{\bf z}}

\def\h{{\bf h}}

\def\be{\begin{equation}}
\def\ee{\end{equation}}
\def\ba{\left[\begin{array}}
\def\ea{\end{array}\right]}

\def\w{{\bf w}}

\def\v{{\bf v}}
\def\x{{\bf x}}
\def\y{{\bf y}}
\def\z{{\bf z}}

\def\1{{\bf 1}}

\def\W{{\bf W}}

\def\g{{\bf g}}
\def\0{{\bf 0}}

\def\htheta{\hat{\theta}}

\def\H{{\bf H}}
\def\X{{\bf X}}

\def\gammainc{\gamma_{inc}}

\def\htheta{\hat{\theta}}

\def\Sw{S_w}

\def\hw{\bar{\h}}

\def\Sw{S_w}

\def\hw{\bar{\H}}

\def\betaweak{\beta_{w}}
\def\thetaweak{\theta_{w}}
\def\hthetaweak{\hat{\theta}_{w}}

\newtheorem{theorem}{Theorem}

\newtheorem{lemma}{Lemma}

\begin{document}

\begin{singlespace}

\title {Random linear under-determined systems with block-sparse solutions -- asymptotics, large deviations, and finite dimensions %A tight variant of Gordon's escape through a mesh theorem
%\footnote{ This work was supported in
%part.}
}
\author{
\textsc{Mihailo Stojnic \footnote
{e-mail: {\tt flatoyer@gmail.com}} }}
\date{}
\maketitle

\centerline{{\bf Abstract}} \vspace*{0.1in}

In this paper we consider random linear under-determined systems with block-sparse solutions. A standard subvariant of such systems, namely, precisely the same type of systems without additional block structuring requirement, gained a lot of popularity over the last decade. This is of course in first place due to the success in mathematical characterization of an $\ell_1$ optimization technique typically used for solving such systems, initially achieved in \cite{CRT,DOnoho06CS} and later on perfected in \cite{DonohoPol,DonohoUnsigned,StojnicCSetam09,StojnicUpper10}. The success that we achieved in \cite{StojnicCSetam09,StojnicUpper10} characterizing the standard sparse solutions systems, we were then able to replicate in a sequence of papers \cite{StojnicCSetamBlock09,StojnicUpperBlock10,StojnicICASSP09block,StojnicJSTSP09} where instead of the standard $\ell_1$ optimization we utilized its an $\ell_2/\ell_1$ variant as a better fit for systems with block-sparse solutions. All of these results finally settled the so-called threshold/phase transitions phenomena (which naturally assume the asymptotic/large dimensional scenario). Here, in addition to a few novel asymptotic considerations, we also try to raise the level a bit, step a bit away from the asymptotics, and consider the finite dimensions scenarios as well.

\vspace*{0.25in} \noindent {\bf Index Terms: Linear systems of equations; block-sparse solutions;
$\ell_2/\ell_1$-optimization}.

\end{singlespace}

%%%%%%%%%%%%%%%%%%%%%%%%%%%%%%%%%%%%%%%%%%%%%%%%%%%%%%%%%%%%%%%%%
\section{Introduction}
\label{sec:back}
%%%%%%%%%%%%%%%%%%%%%%%%%%%%%%%%%%%%%%%%%%%%%%%%%%%%%%%%%%%%%%%%%

We will start by introducing key mathematical structures that we will need in the rest of the paper. However, we do emphasize right here at the beginning that throughout the presentation we will often assume a high level of familiarity with many well known concepts (we will however try to maintain as much consistency with \cite{StojnicCSetamBlock09,StojnicUpperBlock10,StojnicICASSP09block,StojnicJSTSP09} as possible so that one can consult these earlier works and results obtained therein without much need for additional adjustments).

Let system matrix $A$ be an $M\times N$ ($M\leq N$) dimensional matrix with real entries and let $\tilde{\x}$ be an $N$ dimensional vector that also has real entries. Additionally, let $\tilde{\x}$ have no more than $K$ nonzero entries (one usually calls such a vector $K$-sparse). Then one forms the product of $A$ and $\tilde{\x}$ to obtain $\y$
\begin{equation}
\y=A\tilde{\x}. \label{eq:defy}
\end{equation}
The problems of interest in this paper (and in general in linear systems known to have sparse solutions) are essentially an inverted version of (\ref{eq:defy}). Namely, given $A$ and $\y$ from (\ref{eq:defy}), can one then determine $\tilde{\x}$? Mathematically, one asks for the $k$ sparse solution of
\begin{equation}
A\x=\y, \label{eq:system}
\end{equation}
%\begin{figure}[htb]
%%%%%%\begin{minipage}[b]{1.0\linewidth}
%\centering
%\centerline{\epsfig{figure=model.eps,width=10.5cm,height=6cm}}
%%%%%%%\end{minipage}
%\caption{Model of a linear system; vector $\x$ is $k$-sparse}
%\label{fig:model}
%\end{figure}
knowing of course (based on (\ref{eq:defy})) that such a solution exists (for $K<M/2$ it is in fact unique; we also additionally assume that there is no $\x$ in (\ref{eq:system}) that is less than $K$ sparse). Alternatively, one often rewrites the problem described above in the following way
\begin{eqnarray}
\mbox{min} & & \|\x\|_{0}\nonumber \\
\mbox{subject to} & & A\x=\y, \label{eq:l0}
\end{eqnarray}
where $\|\x\|_{0}$ is what is typically called $\ell_0$ norm of vector $\x$ (for all practical purposes needed here, we will think of $\|\x\|_{0}$ as being a number that is equal to the number of nonzero entries of $\x$.

%We also adopt from \cite{StojnicCSetamBlock09,StojnicUpperBlock10,StojnicICASSP09block,StojnicJSTSP09}
%so-called \emph{linear} regime, i.e. we assume that $K=\beta N$
%and that the number of equations is $M=\alpha N$ where
%$\alpha$ and $\beta$ are constants independent of $N$ (more
%on the non-linear regime, i.e. on the regime when $M$ is larger than
%linearly proportional to $K$ can be found in e.g.
%\cite{CoMu05,GiStTrVe06,GiStTrVe07}). Of course, we do mention that all of our results can easily be adapted to various nonlinear regimes as well.

Now that we described the problem one then wonders how it can be solved. Quite a few very successful algorithms have been developed over last several decades (see, e.g. \cite{JATGomp,JAT,NeVe07,DTDSomp,NT08,DaiMil08,DonMalMon09}). In our view however, the most important one is the following linear programming relaxation of (\ref{eq:l0}), called $\ell_1$-optimization
\begin{eqnarray}
\mbox{min} & & \|\x\|_{1}\nonumber \\
\mbox{subject to} & & A\x=\y. \label{eq:l1}
\end{eqnarray}
Its importance is of course rooted in its excellent performance characteristics. However, we do believe that its performance characterizations initially done in \cite{CRT,DOnoho06CS} and later on perfected in \cite{DonohoPol,DonohoUnsigned,StojnicCSetam09,StojnicUpper10} also contributed to a large degree to its an overall popularity. In fact, we believe that \cite{CRT,DOnoho06CS} in particular generated a substantial portion of the interest in linear systems over the last decade. As the main concern of this paper is a bit more specific type of linear systems we will stop short of reviewing in detail results that relate to (\ref{eq:system}) and to (\ref{eq:l1}). Instead we will refer to \cite{DonohoPol,DonohoUnsigned,StojnicCSetam09,StojnicUpper10} and in particular to \cite{StojnicReDirChall13} where a whole lot more can be found.

As mentioned above, in this paper, we will be interested in a bit more specific version of the problem from (\ref{eq:system}). Namely, we will consider the linear systems with solutions that are block-sparse (more on systems with these types of solutions and their potential applications and recovery algorithms can be found in a series of recent references, see e.g. \cite{EldBol09,EKB09,SPH,FHicassp,EMsub,BCDH08,StojnicICASSP09block,StojnicJSTSP09,GaZhMa09,CeInHeBa09} as well as e.g. \cite{ZeGoAd09,ZeWaSeGoAd08,TGS05,BWDSB05,CH06,CREKD,MEldar,Temlyakov04,VPH,BerFri09,EldRau09,BluDav09,NegWai09} and many
references therein for a related problem of recovering jointly sparse vectors). To facilitate the description of the block-sparse vectors we will assume that integers $N$ and $d$ are chosen such that $n=\frac{N}{d}$ is an
integer and it represents the total number of blocks that $\x$
consists of. Clearly $d$ is the length of each block (typically we will assume $d\geq 2$ to distinguish between the standard version of the problem from (\ref{eq:system})). Furthermore,
we will assume that $\X_i=\x_{(i-1)d+1:id}, 1\leq i\leq n$, are the $n$ blocks of $\x$. One can sometimes also assume that $m=\frac{M}{d}$ is an integer as well; however, we will specifically emphasize if/when we do that.
%\begin{figure}[htb]
%%%%%%\begin{minipage}[b]{1.0\linewidth}
%\centering
%\centerline{\epsfig{figure=SysModel.eps,width=12cm,height=10cm}}
%%%%%%%\end{minipage}
%\vspace{-0.55in} \caption{Block-sparse model} \label{fig:blspmodel}
%\end{figure}
We will call any vector $\x$ k-block-sparse if its at most
$k=\frac{K}{d}$ blocks $\X_i$ are non-zero (non-zero block is a block
that is not a zero block; zero block is a block that has all
elements equal to zero). Since $k$-block-sparse signals are
$K$-sparse one could then use (\ref{eq:l1}) to recover the solution
of (\ref{eq:system}). Moreover, all of known results related to performance characterization of (\ref{eq:l1}) would translate immediately as well. However, using (\ref{eq:l1}) to recover block-sparse vectors would utilized their block structure in no way. Since this knowledge is a priori available one would be tempted to believe that it could be somehow utilized in the design of the recovery algorithms so that they outperform the standard $\ell_1$ from (\ref{eq:l1}). This was, of course observed, long time ago and there are of course many ways how one can attempt to exploit the block-sparse structure. Here we will focus on the following polynomial-time algorithm (essentially a
combination of $\ell_2$ and $\ell_1$ optimizations) that was considered in \cite{SPH} (see also e.g. \cite{ZeWaSeGoAd08,ZeGoAd09,BCDH08,EKB09,BerFri09})
\begin{eqnarray}
\mbox{min} & & \sum_{i=1}^{n}\|\x_{(i-1)d+1:id}\|_2\nonumber \\
\mbox{subject to} & & A\x=\y. \label{eq:l2l1}\vspace{-.1in}
\end{eqnarray}
Various aspects of the performance analysis of the optimization problem from (\ref{eq:l2l1}) will be the main subject of the remaining sections of this paper. We will start by reviewing a few known asymptotic results related to the performance characterization of (\ref{eq:l2l1}). We will then present a few novel asymptotic considerations and will finish things off by providing a collection of results in a non-asymptotic setup.

%%%%%%%%%%%%%%%%%%%%%%%%%%%%%%%%%%%%%%%%%%%%%%%%%%%%%%%%%%%%%%%%%
\section{Asymptotics}
\label{sec:asymptotics}
%%%%%%%%%%%%%%%%%%%%%%%%%%%%%%%%%%%%%%%%%%%%%%%%%%%%%%%%%%%%%%%%%

In \cite{SPH} we started studying (\ref{eq:l2l1}) as a tool for solving (\ref{eq:system}). We conducted a serious of numerical experiments as well as a preliminary theoretical analysis of certain properties of (\ref{eq:l2l1}). Through numerical experiments we
demonstrated that as $d$ grows
the algorithm in (\ref{eq:l2l1}) significantly outperforms the
standard $\ell_1$. Of course, as hinted above, just based on the structure of (\ref{eq:l2l1}) it was to be expected that it will be working better than just plain simple $\ell_1$ from (\ref{eq:l1}). What was perhaps a bit surprising (and very welcomed) was that the improvement over $\ell_1$ was rather substantial. In fact, this was particularly true in the so-called under-sampled scenario (i.e. when $m\ll n$) where for fairly moderate values of block-length $d$ (\ref{eq:l2l1}) was already approaching the theoretically limiting level of performance.

Providing a rigorous mathematical justification for these observations at the time was completely out of reach. Nonetheless, we were able to show the following: let $A$ be an $M\times N$ matrix with a basis of null-space
comprised of i.i.d. Gaussian elements; if
$\alpha=\frac{M}{N}\rightarrow 1$ as $N\rightarrow \infty$ then there is a constant $d$ such
that all $k$-block-sparse signals $\x$ with sparsity $K\leq \beta N,
\beta\rightarrow \frac{1}{2}$, can be recovered with overwhelming
probability by solving (\ref{eq:l2l1}).
The precise relation between
$d$ and how fast $\alpha\longrightarrow 1$ and $\beta\longrightarrow
\frac{1}{2}$ was quantified in \cite{SPH} as well. The result was obtained for a fairly narrow range of problem parameters but was in its nature optimal. We of course then recognized the overall potential of (\ref{eq:l2l1}) and undertook a thorough and systematic study of its performance characteristics. In \cite{StojnicICASSP09block,StojnicJSTSP09} we extended the results from
\cite{SPH} and obtained the values of the recoverable block-sparsity for any
$\alpha$, i.e. for $0\leq \alpha \leq 1$. More precisely, for any
given constant $0\leq \alpha \leq 1$ we in \cite{StojnicICASSP09block,StojnicJSTSP09} determined a constant
$\beta=\frac{K}{N}$ such that for a sufficiently large $d$ (\ref{eq:l2l1})
with overwhelming probability
recovers any $k$-block-sparse vector of sparsity less than $K$
(in this paper we, as usual, under overwhelming probability assume
a probability that is no more than a number exponentially decaying in $N$ away from $1$).

%Clearly, for any given constant $\alpha\leq 1$ there is a maximum
%allowable value of the constant $\beta$ such that (\ref{eq:l2l1})
%finds solution of (\ref{eq:system}) with overwhelming probability
%for \emph{any} $K$-block-sparse $\x$. This maximum allowable value of the constant
%$\beta$ is called the \emph{strong threshold} (see
%\cite{DonohoUnsigned,DonohoPol}). We will denote the value of the strong
%threshold by $\beta_s$. Similarly, for any given constant
%$\alpha\leq 1$ one can define the \emph{sectional threshold} as the
%maximum allowable value of the constant $\beta$ such that
%(\ref{eq:l2l1}) finds the solution of (\ref{eq:system}) with overwhelming
%probability for \emph{any} $K$-block-sparse $\x$ with a given fixed location of non-zero blocks (see \cite{DonohoUnsigned,DonohoPol}).
%In a similar fashion one can then denote the value of the sectional threshold by $\beta_{sec}$. Finally, for any given constant
%$\alpha\leq 1$ one can define the \emph{weak threshold} as the
%maximum allowable value of the constant $\beta$ such that
%(\ref{eq:l2l1}) finds the solution of (\ref{eq:system}) with overwhelming
%probability for \emph{any} $K$-block-sparse $\x$ with a given fixed location of non-zero blocks and given fixed directions of non-zero block vectors $\X_i$ (see \cite{DonohoUnsigned,DonohoPol}).
%In a similar fashion one can then denote the value of the weak threshold by $\beta_{w}$.

For algorithms that exhibit the so-called phase-transition (PT) phenomenon, for any given constant $\alpha\leq 1$ there is a maximum
allowable value of $\beta$ such that for \emph{any} given $K$-sparse $\x$ in (\ref{eq:system}) the solution that the algorithm offers
is with overwhelming probability exactly that given $K$-sparse $\x$. To be a bit more precise and in an alignment with some of our earlier works, we should add that this type of phase transitions is often called the \emph{strong} phase transition (the \emph{strong} PT) and the value of
$\beta$ that we just described is typically referred to as the \emph{strong threshold} (see
\cite{DonohoPol,StojnicCSetamBlock09}). These threshold values are essentially the points (proportions of the system dimensions) where the algorithms (in our case here (\ref{eq:l2l1})) exhibit the phase-transition phenomenon. In a bit less formal language, the phase-transition phenomenon essentially means that if the problem dimensions are such that the pair $(\alpha,\beta)$ is below the so called phase-transition curve (the PT curve) then the algorithm (here (\ref{eq:l2l1})) solves the problem (here (\ref{eq:system})); otherwise it fails. The ultimate goal of an asymptotic analysis of any algorithm with phase transition is to determine the PT curve.

We should also mention that the above requirement is, from a practical point of view, a bit restrictive. Namely, insisting that the algorithm succeeds for \emph{any} given $K$-sparse $\x$ is a bit too much to ask for if one only cares about a ``typical" level of performance. To characterize a bit more typical levels of performance one often relaxes the above PT description. Namely, for any given constant
$\alpha\leq 1$ and \emph{any} given $\x$ with a given fixed location and a given fixed directions of non-zero blocks
there will be a maximum allowable value of $\beta$ such that
(\ref{eq:l2l1}) finds that given $\x$ in (\ref{eq:system}) with overwhelming
probability. Such a value of
$\beta$, which we will here denote by $\beta_{w}$, is typically called the \emph{weak threshold} and the resulting curve the \emph{weak} phase-transition curve (see, e.g. \cite{StojnicICASSP09,StojnicCSetam09}). More important than the name itself is that this type of phase-transition is designed to capture a typical performance in a better way. Indeed, when one needs (\ref{eq:system}) solved, it typically wants it solved for an $\x$ or for a set of $\x$ but quite likely not for every single $\x$. In scenarios when this is indeed true the above weak phase-transition curve is highly likely to be more useful.

Now, returning back to what was done in \cite{StojnicICASSP09block,StojnicJSTSP09}, we should emphasize another subtle point. Namely, \cite{StojnicICASSP09block,StojnicJSTSP09} provided fairly sharp strong threshold values but they had done so in a so to say block asymptotic sense, i.e. the analyses presented in \cite{StojnicICASSP09block,StojnicJSTSP09} assumed fairly large values of the block-length $d$. As such they then provided an ultimate performance limit of $\ell_2/\ell_1$-optimization rather than its performance characterization as a function of a particular fixed block-length.

Finally in \cite{StojnicCSetamBlock09} we were able to massively extended the results of \cite{StojnicICASSP09block,StojnicJSTSP09} while moving away from the block-asymptotic regime. We essentially introduced a novel probabilistic framework for performance characterization of (\ref{eq:l2l1}) through which we were finally able to view block-length as a parameter of the system (some of the key components of the framework were actually introduced in \cite{StojnicCSetam09}). Through the framework we were able to lower-bound $\beta_w$. Moreover, the obtained lower bounds were in an excellent agreement with the values obtained for $\beta_w$ through numerical simulations. In \cite{StojnicUpperBlock10} we then showed that the lower bounds on $\beta_w$ obtained in \cite{StojnicCSetamBlock09} are actually exact. We below recall on a theorem that essentially summarizes the results obtained in \cite{StojnicCSetamBlock09,StojnicUpperBlock10} and effectively establishes for any $0<\alpha\leq 1$ the exact value of $\beta_w$ for which (\ref{eq:l2l1}) finds the $k$-block-sparse $\x$ from (\ref{eq:system}).

\begin{theorem}(\cite{StojnicCSetamBlock09,StojnicUpperBlock10} Exact weak threshold; block-sparse $\x$)
Let $A$ be an $M\times N$ matrix in (\ref{eq:system})
with i.i.d. standard normal components. Let
the unknown $\x$ in (\ref{eq:system}) be $k$-block-sparse with the length of its blocks $d$. Further, let the location and the directions of nonzero blocks of $\x$ be arbitrarily chosen but fixed.
Let $k,m,n$ be large
and let $\alpha=\frac{m}{n}$ and $\betaweak=\frac{k}{n}$ be constants
independent of $m$ and $n$. Let $\gammainc(\cdot,\cdot)$ and $\gammainc^{-1}(\cdot,\cdot)$ be the incomplete gamma function and its inverse, respectively. Further,
let all $\epsilon$'s below be arbitrarily small constants.

\begin{enumerate}
\item Let $\hthetaweak$, ($\betaweak\leq \hthetaweak\leq 1$) be the solution of
\begin{equation}
(1-\epsilon_1^{(c)})(1-\betaweak)\frac{\frac{\sqrt{2}\Gamma(\frac{d+1}{2})}{\Gamma(\frac{d}{2})}
\left (1-\gammainc(\gammainc^{-1}(\frac{1-\thetaweak}{1-\betaweak},\frac{d}{2}),\frac{d+1}{2})\right )}{\thetaweak}-\sqrt{2\gammainc^{-1}(\frac{(1+\epsilon_1^{(c)})(1-\thetaweak)}{1-\betaweak},\frac{d}{2})}=0
.\label{eq:thmweaktheta}
\end{equation}
If $\alpha$ and $\betaweak$ further satisfy
\begin{multline}
\alpha d>(1-\betaweak)\frac{2\Gamma(\frac{d+2}{2})}{\Gamma(\frac{d}{2})}
\left (1-\gammainc(\gammainc^{-1}(\frac{1-\hthetaweak}{1-\betaweak},\frac{d}{2}),\frac{d+2}{2})\right )
+\betaweak d\\-\frac{\left ((1-\betaweak)\frac{\sqrt{2}\Gamma(\frac{d+1}{2})}{\Gamma(\frac{d}{2})}
(1-\gammainc(\gammainc^{-1}(\frac{1-\hthetaweak}{1-\betaweak},\frac{d}{2}),\frac{d+1}{2}))\right ) ^2}{\hthetaweak}\label{eq:thmweakalpha}
\end{multline}
then with overwhelming probability the solution of (\ref{eq:l2l1}) is the $k$-block-sparse $\x$ from (\ref{eq:system}).

\item Let $\htheta_w$, ($\beta_w\leq \htheta_w\leq 1$) be the solution of
\begin{equation}
(1+\epsilon_2^{(c)})(1-\betaweak)\frac{\frac{\sqrt{2}\Gamma(\frac{d+1}{2})}{\Gamma(\frac{d}{2})}
\left (1-\gammainc(\gammainc^{-1}(\frac{1-\thetaweak}{1-\betaweak},\frac{d}{2}),\frac{d+1}{2})\right )}{\thetaweak}-\sqrt{2\gammainc^{-1}(\frac{(1-\epsilon_2^{(c)})(1-\thetaweak)}{1-\betaweak},\frac{d}{2})}=0
.\label{eq:thmweaktheta1}
\end{equation}
If $\alpha$ and $\betaweak$ further satisfy
\begin{multline}
\alpha d<\frac{1}{(1+\epsilon_1^{(m)})^2}((1-\epsilon_1^{(g)}) (1-\betaweak)\frac{2\Gamma(\frac{d+2}{2})}{\Gamma(\frac{d}{2})}
\left (1-\gammainc(\gammainc^{-1}(\frac{1-\hthetaweak}{1-\betaweak},\frac{d}{2}),\frac{d+2}{2})\right )
+\betaweak d\\-\frac{\left ((1-\betaweak)\frac{\sqrt{2}\Gamma(\frac{d+1}{2})}{\Gamma(\frac{d}{2})}
(1-\gammainc(\gammainc^{-1}(\frac{1-\hthetaweak}{1-\betaweak},\frac{d}{2}),\frac{d+1}{2}))\right ) ^2}{\hthetaweak (1+\epsilon_3^{(g)})^{-2}} )\label{eq:thmweakalpha1}
\end{multline}
then with overwhelming probability there will be a $k$-block-sparse $\x$ (from a set of $\x$'s with fixed locations and directions of nonzero blocks) that satisfies (\ref{eq:system}) and is \textbf{not} the solution of (\ref{eq:l2l1}).
\end{enumerate}
\label{thm:thmweakthrblock}
\end{theorem}
\begin{proof}
The first part was established in \cite{StojnicCSetamBlock09}. The second part was established in \cite{StojnicUpperBlock10}.
\end{proof}

To give a little bit of flavor as to what is actually proven in the above theorem we in Figure \ref{fig:weak} show the theoretical thresholds one obtains based on the above theorem. Moreover, in \cite{StojnicCSetamBlock09,StojnicUpperBlock10} we showed that the theoretical prediction given in Figure \ref{fig:weak} (and obtained in asymptotic infinite-dimensional scenario) are also in a solid agreement with what the finite dimensional numerical simulations can produce (in fact the agreement is already fairly good for the block-length as small as $100$).

\begin{figure}[htb]
%\begin{minipage}[b]{.5\linewidth}
\centering
\centerline{\epsfig{figure=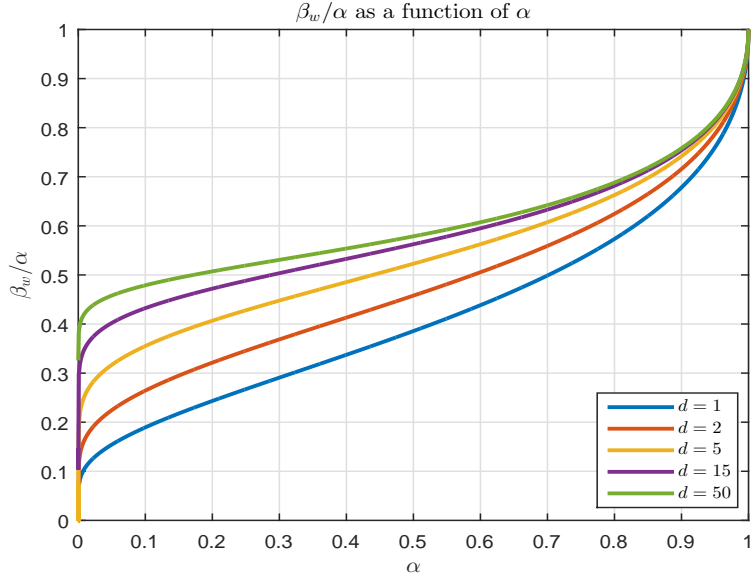,width=11.5cm,height=8cm}}
%\end{minipage}
%\begin{minipage}[b]{.5\linewidth}
%\centering
%\centerline{\epsfig{figure=finprerral08.eps,width=9cm,height=6.5cm}}
%\end{minipage}
\caption{Block-sparse weak PT curves}
\label{fig:weak}
\end{figure}

%\begin{figure}[htb]
%\begin{minipage}[b]{.5\linewidth}
%\centering
%\centerline{\epsfig{figure=CSetamBlockWeak.eps,width=7.5cm,height=7cm}}
%%\end{minipage}
%%\begin{minipage}[b]{.5\linewidth}
%%\centering
%%\centerline{\epsfig{figure=finprerral08.eps,width=9cm,height=6.5cm}}
%\end{minipage}
%\begin{minipage}[b]{.5\linewidth}
%\centering
%\centerline{\epsfig{figure=SimulBlSpWeakd151.eps,width=7.5cm,height=7cm}}
%%\end{minipage}
%%\begin{minipage}[b]{.5\linewidth}
%%\centering
%%\centerline{\epsfig{figure=finprerral08.eps,width=9cm,height=6.5cm}}
%\end{minipage}
%\caption{\emph{Weak} threshold, $\ell_2/\ell_1$-optimization; theory -- left, simulations -- right}
%\label{fig:weak}
%\end{figure}

%%%%%%%%%%%%%%%%%%%%%%%%%%%%%%%%%%%%%%%%%%%%%%%%%%%%%%%%%%%%%%%%%
\section{Large deviations}
\label{sec:ldp}
%%%%%%%%%%%%%%%%%%%%%%%%%%%%%%%%%%%%%%%%%%%%%%%%%%%%%%%%%%%%%%%%%

In the previous section we recalled on a collection of results that deal with the so-called phase-transition phenomenon. In this section we will introduce a somewhat novel concept that effectively resembles what is in probability theory known as the large deviation principle (LDP). We start by recalling on a couple of results that we established in \cite{StojnicCSetamBlock09}.

For the simplicity and clarity of the exposition and without loss of generality we will assume that the blocks $\X_{1},\X_{2},\dots,\X_{n-k}$ of $\x$ are equal to zero and that that vectors $\X_{n-k+1},\X_{n-k+2},\dots,\X_n$ have fixed directions. Furthermore, since all probability distributions of interest will be rotationally invariant we will later assume that $\X_i=(\|\X_i\|_2,0,0,\dots,0),n-k+1\leq i \leq n$.  The following was proved in \cite{StojnicCSetamBlock09}.
\begin{theorem}(\cite{StojnicCSetamBlock09} Nonzero blocks of $\x$ have fixed directions and location)
Assume that an $M\times dn$ measurement matrix $A$ is given. Let $\x$
be a $k$-block-sparse vector. Also let $\X_1=\X_2=\dots=\X_{n-k}=0$. Let the directions of vectors $\X_{n-k+1},\X_{n-k+2},\dots,\X_n$ be fixed. Further, assume that $\y=A\x$ and that $\w$ is
a $dn\times 1$ vector. Then (\ref{eq:l2l1}) will
produce the solution of (\ref{eq:system}) if
\begin{equation}
(\forall \w\in \textbf{R}^{dn} | A\w=0) \quad  -\sum_{i=n-k+1}^n \frac{\X_i^T\W_i}{\|\X_i\|_2}<\sum_{i=1}^{n-k}\|\W_{i}\|_2.
\label{eq:thmeqgenweak}
\end{equation}\label{thm:thmgenweak}
\end{theorem}
To facilitate the exposition we set
\begin{equation}
\Sw'=\{\w\in S^{dn-1}| \quad -\sum_{i=n-k+1}^n \frac{\X_i^T\W_i}{\|\X_i\|_2}<\sum_{i=1}^{n-k}\|\W_{i}\|_2\}.\label{eq:defSwpr}
\end{equation}

%%%%%%%%%%%%%%%%%%%%%%%%%%%%%%%%%%%%%%%%%%%%%%%%%%%%%%%%%%%%%%%%%
\subsection{Upper tail}
\label{sec:uppertail}
%%%%%%%%%%%%%%%%%%%%%%%%%%%%%%%%%%%%%%%%%%%%%%%%%%%%%%%%%%%%%%%%%

Assuming as earlier that the elements of $A$ are i.i.d. standard normals, our object of interest in this section will be the following probability
\begin{equation}
P_{err}= P(\min_{\w\in S_w'}\|A\w\|_2\leq 0)=P(\max_{\w\in S_w'}\min_{\|\y\|_2=1}(\y^T A\w )\geq 0).
\label{eq:ldpprob}
\end{equation}
Clearly, $P_{err}$ is the probability that (\ref{eq:l2l1}) fails to produce the solution of (\ref{eq:system}). Using the Chernoff bound we then get
\begin{equation}
P_{err}=P(\max_{\w\in S_w'}\min_{\|\y\|_2=1}(\y^T A\w )\geq 0)\leq E e^{c_3\max_{\w\in S_w'}\min_{\|\y\|_2=1}(\y^T A\w )}=E \max_{\w\in S_w'}\min_{\|\y\|_2=1}e^{(-c_3\y^T A\w )},
\label{eq:ldpprob1}
\end{equation}
where we assume $c_3\geq 0$ and note that $A$ and $-A$ have the same distribution. We will then rely on the following lemma utilized in \cite{StojnicMoreSophHopBnds10} (as discussed in \cite{StojnicMoreSophHopBnds10}, the lemma follows from a comparison result from \cite{Gordon85}).
\begin{lemma}(\cite{Gordon85,StojnicMoreSophHopBnds10})
Let $A$ be an $M\times N$ matrix with i.i.d. standard normal components. Let $\g$ and $\h$ be $M\times 1$ and $N\times 1$ vectors, respectively, with i.i.d. standard normal components. Also, let $g$ be a standard normal random variable and let $c_3$ be a positive constant. Then
\begin{equation}
E(\max_{\w\in S_w'}\min_{\|\y\|_2=1}e^{-c_3(\y^T A\w + g)})\leq E(\max_{\w\in S_w'}\min_{\|\y\|_2=1}e^{-c_3(\g^T\y+\h^T\w)}).\label{eq:negexplemma}
\end{equation}\label{lemma:negexplemma}
\end{lemma}
Transforming (\ref{eq:ldpprob}) a bit we obtain
\begin{equation}
P_{err}\leq E \max_{\w\in S_w'}\min_{\|\y\|_2=1}e^{-c_3\y^T A\w }=e^{-\frac{c_3^2}{2}}E e^{c_3g}E\max_{\w\in S_w'}\min_{\|\y\|_2=1}e^{-c_3\y^T A\w}
=e^{-\frac{c_3^2}{2}}E\max_{\w\in S_w'}\min_{\|\y\|_2=1}e^{-c_3(\y^T A\w+g)}.
\label{eq:ldpprob2}
\end{equation}
Now a combination of (\ref{eq:negexplemma}) and  (\ref{eq:ldpprob2}) gives
\begin{multline}
P_{err}\leq e^{-\frac{c_3^2}{2}}E\max_{\w\in S_w'}\min_{\|\y\|_2=1}e^{-c_3(\y^T A\w+g)}\leq
e^{-\frac{c_3^2}{2}}E\max_{\w\in S_w'}\min_{\|\y\|_2=1}e^{-c_3(\g^T\y+\h^T\w)}\\
=e^{-\frac{c_3^2}{2}}Ee^{-c_3\|\g\|_2}E\max_{\w\in S_w'}e^{-c_3\h^T\w}=e^{-\frac{c_3^2}{2}}Ee^{-c_3\|\g\|_2}Ee^{c_3\max_{\w\in S_w'}\h^T\w}
=e^{-\frac{c_3^2}{2}}Ee^{-c_3\|\g\|_2}Ee^{c_3w(\h,S_w')},
\label{eq:ldpprob3}
\end{multline}
where
\begin{equation}
w(\h,\Sw')=\max_{\w\in \Sw'} (\h^T\w) \label{eq:widthdefSwpr}
\end{equation}
and we also note that $\h$ and $-\h$ have the same distribution. We now set
\begin{equation}
\Sw=\{\w\in S^{dn-1}| \quad -\sum_{i=n-k+1}^n \w_{(i-1)d+1}<\sum_{i=1}^{n-k}\|\W_{i}\|_2\}\label{eq:defSw}
\end{equation}
and
\begin{equation}
w(\h,\Sw)=\max_{\w\in \Sw} (\h^T\w). \label{eq:widthdefSw}
\end{equation}
Similarly to what was argued in \cite{StojnicCSetamBlock09} we have
\begin{equation}
P_{err}\leq e^{-\frac{c_3^2}{2}}Ee^{-c_3\|\g\|_2}Ee^{c_3w(\h,S_w')}=e^{-\frac{c_3^2}{2}}Ee^{-c_3\|\g\|_2}Ee^{c_3w(\h,S_w)}.
\label{eq:ldpprob4}
\end{equation}
For a moment we will now step away from $P_{err}$ and will instead focus on $w(\h,S_w)$. As in \cite{StojnicCSetamBlock09} we have
\begin{equation}
w(\h,\Sw)=\max_{\w\in \Sw} (\h^T\w)=\max_{\w\in \Sw} (\sum_{i=n-k+1}^n \h_{(i-1)d+1}\w_{(i-1)d+1}+
\sum_{i=n-k+1}^n\|\H_i^*\|_2\|\W_i^*\|_2+\sum_{i=1}^{n-k}\|\H_i\|_2\|\W_i\|_2),\label{eq:workww0}
\end{equation}
where
\begin{eqnarray}
  \H_i   &=& (\h_{(i-1)d + 1}, \h_{(i-1)d + 2}, \ldots, \h_{id})^T, i = 1, 2, \ldots, n-k \nonumber \\
  \H_i^* &=& (\h_{(i-1)d + 2}, \h_{(i-1)d + 3}, \ldots, \h_{id})^T, i = n-k+1, n-k+2, \ldots, n \nonumber \\
  \W_i^* &=& (\w_{(i-1)d + 2}, \w_{(i-1)d + 3}, \ldots, \w_{id})^T, i = n-k+1, n-k+2, \ldots, n.\label{eq:workww00}
\end{eqnarray}
Set
\begin{multline}
\hw=(\|\H_1\|_2,\|\H_2\|_2,\dots,\|\H_{n-k}\|_2,
-\h_{(n-k)d+1},-\h_{(n-k+1)d+1},-\h_{(n-k+2)d+1},\dots,-\h_{(n-1)d+1},\\
\|\H_{n-k+1}^*\|_2,\|\H_{n-k+2}^*\|_2,\dots,\|\H_{n}^*\|_2)^T.\label{eq:defhweak}
\end{multline}
Then one can simplify (\ref{eq:workww0}) in the following way
\begin{eqnarray}
w(\h,\Sw) = \max_{\bar{\y}\in R^{n+k}} & &  \sum_{i=1}^{n+k} \hw_i \bar{\y}_i\nonumber \\
\mbox{subject to} &  & \bar{\y}_i\geq 0, 0\leq i\leq n-k,n+1\leq i\leq n+k\nonumber \\
& & \sum_{i=n-k+1}^n\bar{\y}_i\geq \sum_{i=1}^{n-k} \bar{\y}_i \nonumber \\
& & \sum_{i=1}^{n+k}\bar{\y}_i^2\leq 1\label{eq:workww2}
\end{eqnarray}
where $\hw_i$ is the $i$-th element of $\hw$ and $\bar{\y}_i$ is the $i$-th element of $\bar{\y}$. Utilizing the machinery of \cite{StojnicCSetam09,StojnicCSetamBlock09} one then has
\begin{eqnarray}
w(\h,\Sw) = -\max_{\lambda\geq 0,\gamma\geq 0}\min_{\bar{\y}} & & \sum_{i=1}^{n+k} -\hw_i \bar{\y}_i+\lambda\sum_{i=1}^{n-k}\bar{\y}_i
-\lambda\sum_{i=n-k+1}^{n}\bar{\y}_i+\gamma\sum_{i=1}^{n+k}\bar{\y}_i^2-\gamma\nonumber \\
\mbox{subject to} & & \bar{\y}_i\geq 0, 0\leq i\leq n-k,n+1\leq i\leq n+k.\label{eq:ldpwhSw0}
\end{eqnarray}
After solving the inner minimization one finally obtains
\begin{eqnarray}
w(\h,\Sw) & = & \min_{\lambda\geq0,\gamma\geq 0} \frac{\sum_{i=1}^{n-k}\max(\hw_i-\lambda,0)^2+\sum_{i=n-k+1}^{n}(\hw_i+\lambda)^2+\sum_{i=n+1}^{n+k}\hw_i^2}{4\gamma}+\gamma\nonumber \\
& = & \min_{\lambda\geq0}\sqrt{\sum_{i=1}^{n-k}\max(\hw_i-\lambda,0)^2+\sum_{i=n-k+1}^{n}(\hw_i+\lambda)^2+\sum_{i=n+1}^{n+k}\hw_i^2}.\label{eq:ldpwhSw}
\end{eqnarray}
A combination of (\ref{eq:ldpprob4}) and (\ref{eq:ldpwhSw}) provides an upper bound on $P_{err}$. We summarize these results in the following theorem.
\begin{theorem}
Let $A$ be an $M\times N$ matrix in (\ref{eq:system})
with i.i.d. standard normal components. Let
the unknown $\x$ in (\ref{eq:system}) be $k$-block-sparse with the length of its blocks $d$. Further, let the location and the directions of nonzero blocks of $\x$ be arbitrarily chosen but fixed. Let $P_{err}$ be the probability that the solution of (\ref{eq:l2l1}) is not the $k$-block-sparse solution of (\ref{eq:system}). Then
\begin{equation}
P_{err}\leq \min_{c_3\geq 0}e^{-\frac{c_3^2}{2}}e^{-c_3\|\g\|_2}Ee^{c_3w(\h,S_w')}
=\min_{c_3\geq 0}\left (e^{-\frac{c_3^2}{2}}\frac{1}{\sqrt{2\pi}^M}\int_{\g}e^{-\sum_{i=1}^{M}\g_i^2/2-c_3\|\g\|_2}d\g \min_{\lambda\geq 0,\gamma\geq\frac{c_3}{2}} w_1^{n-k}w_2^{k}w_3^ke^{c_3\gamma}\right ),
\label{eq:ldpthm1perrub1}
\end{equation}
where
\begin{eqnarray}
  w_1 &=& \int_{\hw_i\geq 0}\frac{2^{-d/2}}{\Gamma(d/2)}\hw_i^{d/2-1}e^{-\hw_i/2}e^{c_3\max(\sqrt{\hw_i}-\lambda,0)^2/4/\gamma}d\hw_i\nonumber \\
  w_2 &=& \frac{1}{\sqrt{2\pi}}\int_{\hw_i}e^{-\hw_i^2/2}e^{c_3(\hw_i+\lambda)^2/4/\gamma}d\hw_i\nonumber \\
  w_3 &=& \int_{\hw_i\geq 0}\frac{2^{-(d-1)/2}}{\Gamma((d-1)/2)}\hw_i^{(d-1)/2-1}e^{-\hw_i/2}e^{c_3\hw_i/4/\gamma}d\hw_i.\label{eq:ldpthm1perrub2}
\end{eqnarray}\label{thm:ldp1}
\end{theorem}
\begin{proof}
Follows by combining (\ref{eq:ldpprob4}) and (\ref{eq:ldpwhSw}) and noting that $\hw_i,i\in[n-k+1,n]$, are i.i.d. standard normals and  $\hw_i^2,i\notin[n-k+1,n]$, are $\chi$-square distributed according to
\begin{eqnarray}\label{eq:ldpthm1perrub3}
  p(\hw_i^2) &=& \frac{2^{-d/2}}{\Gamma(d/2)}(\hw_i^2)^{d/2-1}e^{-\hw_i^2/2},1\leq i\leq n-k\nonumber   \\
  p(\hw_i) &=& \frac{1}{\sqrt{2\pi}}e^{-\hw_i^2/2},n-k+1\leq i\leq n\nonumber   \\
  p(\hw_i^2) &=& \frac{2^{-(d-1)/2}}{\Gamma((d-1)/2)}(\hw_i^2)^{(d-1)/2-1}e^{-\hw_i^2/2},n+1\leq i\leq n+k,
  \end{eqnarray}
where, as usual, $\Gamma(\cdot)$ stands for the standard gamma function.
\end{proof}
We emphasize here that such a bound is valid for any integers $M$, $k$, $n$, and $d$ (of course assuming $dk\leq M\leq dn$ so that the results make sense). As in the previous section, in this section we are also interested in the asymptotic scenario, namely the scenario from Theorem \ref{thm:thmgenweak}. In particular, we are interested in the scenarios where $\alpha=\frac{m}{n}\geq (1+\epsilon_{\alpha})\alpha_w$ with $\epsilon_{\alpha}>0$ and $\alpha_w$ being the minimal $\alpha$ such that for a fixed $\beta_w$ the equations of Theorem \ref{thm:thmgenweak} are satisfied. In such a scenario what primarily determines $P_{err}$ is its exponent. Hence, one has
\begin{equation}\label{ldpasymp1}
  I_{err}=\lim_{n\rightarrow\infty}\frac{\log{P_{err}}}{n}.
\end{equation}
This of course clearly resembles the so-called large deviation principle (LDP) with $I_{err}$ emulating the so-called LDP's rate (indicator) function. Based on Theorem \ref{thm:ldp1} we have the following LDP type of theorem.
\begin{theorem}
Assume the setup of Theorem \ref{thm:ldp1}. Further, let integers $M$, $k$, and $n$ be large ($dk\leq M\leq dn$). Also, assume the following scaling: $c_3=c_3^{(s)}\sqrt{n}$ and $\gamma=\gamma^{(s)}\sqrt{n}$. Then
\begin{equation}
I_{err}=\lim_{n\rightarrow\infty}\frac{\log{P_{err}}}{n}
\leq \min_{c_3^{(s)}\geq 0}\left (-\frac{(c_3^{(s)})^2}{2}+I_{sph}+\min_{\lambda\geq 0,\gamma^{(s)}\geq 0} ((1-\beta_w)\log{w_1}+\beta_w\log{w_2}+\beta_w\log{w_3}+c_3^{(s)}\gamma^{(s)})\right ),
\label{eq:ldpthm2Ierrub1}
\end{equation}
where
\begin{eqnarray}
I_{sph} &=& \widehat{\gamma^{(s)}}c_3^{(s)}-\frac{\alpha d}{2}\log\left (1-\frac{c_3^{(s)}}{2\widehat{\gamma^{(s)}}}\right )\nonumber \\
  \widehat{\gamma^{(s)}} &=& \frac{2c_3^{(s)}-\sqrt{4(c_3^{(s)})^2+16\alpha d}}{8}\nonumber \\
  w_1 &=& \int_{\hw_i\geq 0}\frac{2^{-d/2}}{\Gamma(d/2)}\hw_i^{d/2-1}e^{-\hw_i/2}e^{c_3^{(s)}\max(\sqrt{\hw_i}-\lambda,0)^2/4/\gamma^{(s)}}d\hw_i\nonumber \\
  w_2 &=& \frac{1}{\sqrt{2\pi}}\int_{\hw_i}e^{-\hw_i^2/2}e^{c_3^{(s)}(\hw_i+\lambda)^2/4/\gamma^{(s)}}d\hw_i=e^{(c_3^{(s)}/4/\gamma^{(s)}\lambda^2)/(1-2c_3^{(s)}/4/\gamma^{(s)})}/\sqrt{1-2c_3^{(s)}/4/\gamma^{(s)}}\nonumber \\
  w_3 &=& \int_{\hw_i\geq 0}\frac{2^{-(d-1)/2}}{\Gamma((d-1)/2)}\hw_i^{(d-1)/2-1}e^{-\hw_i/2}e^{c_3^{(s)}\hw_i/4/\gamma^{(s)}}d\hw_i=\left (1/\sqrt{1-2c_3^{(s)}/4/\gamma^{(s)}}\right )^{d-1}.\label{eq:ldpthm2perrub2}
\end{eqnarray}\label{thm:ldp2}
\end{theorem}
\begin{proof} Follows from Theorem \ref{thm:ldp1} and by noting that in \cite{StojnicMoreSophHopBnds10} we established
\begin{equation}
I_{sph}=\lim_{n\rightarrow\infty}\frac{1}{n}\log(Ee^{-c_3^{(s)}\sqrt{n}\|\g\|_2})
=\widehat{\gamma^{(s)}}c_3^{(s)}-\frac{\alpha d}{2}\log\left (1-\frac{c_3^{(s)}}{2\widehat{\gamma^{(s)}}}\right ),\label{eq:gamaiden2lift}
\end{equation}
where
\begin{equation}
\widehat{\gamma^{(s)}}=\frac{2c_3^{(s)}-\sqrt{4(c_3^{(s)})^2+16\alpha d}}{8}.\label{eq:gamaiden3lift}
\end{equation}
\end{proof}
In Figure \ref{fig:ldpIerrubd2d10} we present the estimates of the $P_{err}$ exponents (i.e. of $I_{err}$ -- the LDP's rate function values) that can be obtained using Theorem \ref{thm:ldp2}. Additionally, in Tables \ref{tab:Ierrldpubtab2d2} and \ref{tab:Ierrldpubtab2d10} we present some of the values we used for $c_3^{(s)}$ , $\gamma^{(s)}$, and $\lambda$ to compute $I_{err,u}^{(ub)}$ as the estimate. Plots in Figure \ref{fig:ldpIerrubd2d10} indicate that increasing block-length may shorten the so-called transition zone. A similar conclusion will also appear as reasonable in finite dimension scenarios that we will discuss in the following section.
\begin{table}[h]
\caption{A collection of values for $c_3^{(s)}$ , $\gamma^{(s)}$, $\lambda$, and $I_{err,u}^{(ub)}$ in Theorem \ref{thm:ldp2}; $d=2$, $\beta=\frac{1}{3}$}\vspace{.1in}
\hspace{-0in}\centering
\begin{tabular}{||c||c|c|c|c|c||}\hline\hline
$\alpha$        & $ 0.6500 $ & $ 0.7100 $ & $ 0.7700 $ & $ 0.8300 $ & $ 0.8900 $\\ \hline\hline
$c_3^{(s)}$     & $ 0.1263 $ & $ 0.7007 $ & $ 1.3534 $ & $ 2.2095 $ & $ 3.4521 $\\ \hline
$\gamma^{(s)}$  & $ 0.6026 $ & $ 0.7963 $ & $ 1.0451 $ & $ 1.4010 $ & $ 1.9539 $\\ \hline
$\lambda$       & $ 0.8973 $ & $ 0.7932 $ & $ 0.6965 $ & $ 0.5883 $ & $ 0.4701 $\\ \hline\hline
$I_{err,u}^{(ub)}$& $ \bl{\mathbf{-0.0007}} $ & $ \bl{\mathbf{-0.0215}} $ & $ \bl{\mathbf{-0.0701}} $ & $ \bl{\mathbf{-0.1475}} $ & $ \bl{\mathbf{-0.2565}} $\\ \hline\hline
\end{tabular}
\label{tab:Ierrldpubtab2d2}
\end{table}

\begin{table}[h]
\caption{A collection of values for $c_3^{(s)}$ , $\gamma^{(s)}$, $\lambda$, and $I_{err,u}^{(ub)}$ in Theorem \ref{thm:ldp2}; $d=10$, $\beta=\frac{1}{3}$}\vspace{.1in}
\hspace{-0in}\centering
\begin{tabular}{||c||c|c|c|c|c||}\hline\hline
$\alpha$        & $ 0.5900 $ & $ 0.6400 $ & $ 0.6900 $ & $ 0.7400 $ & $ 0.7900 $\\ \hline\hline
$c_3^{(s)}$     & $ 0.2334 $ & $ 1.1136 $ & $ 2.0377 $ & $ 3.0881 $ & $ 4.3441 $\\ \hline
$\gamma^{(s)}$  & $ 1.2743 $ & $ 1.5736 $ & $ 1.9182 $ & $ 2.3361 $ & $ 2.8622 $\\ \hline
$\lambda$       & $ 2.0307 $ & $ 1.8998 $ & $ 1.7647 $ & $ 1.6125 $ & $ 1.4488 $\\ \hline\hline
$I_{err,u}^{(ub)}$& $ \bl{\mathbf{-0.0032}} $ & $ \bl{\mathbf{-0.0702}} $ & $ \bl{\mathbf{-0.2199}} $ & $ \bl{\mathbf{-0.4498}} $ & $ \bl{\mathbf{-0.7622}} $\\ \hline\hline
\end{tabular}
\label{tab:Ierrldpubtab2d10}
\end{table}

\begin{figure}[htb]
%\begin{minipage}[b]{.5\linewidth}
\centering
\centerline{\epsfig{figure=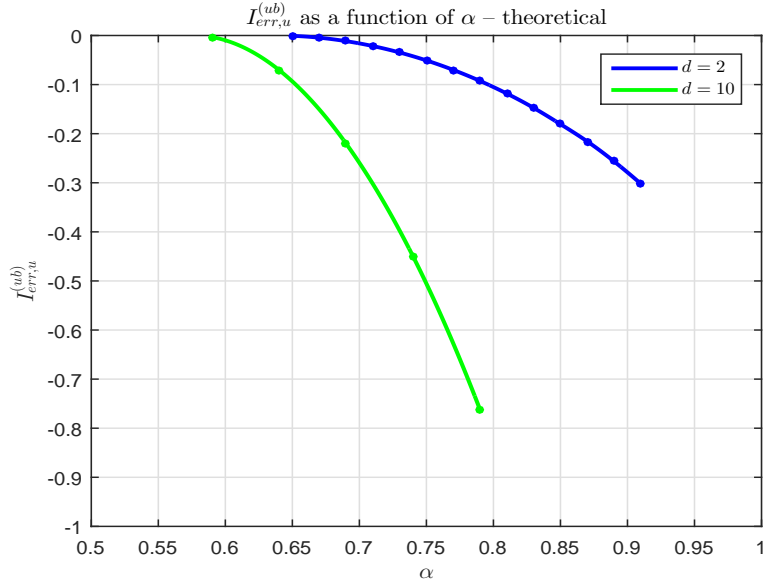,width=11.5cm,height=8cm}}
%\end{minipage}
%\begin{minipage}[b]{.5\linewidth}
%\centering
%\centerline{\epsfig{figure=finprerral08.eps,width=9cm,height=6.5cm}}
%\end{minipage}
\caption{$I_{err,u}^{(ub)}$ as a function of $\alpha$; block-lengths $d=2$ and $d=10$}
\label{fig:ldpIerrubd2d10}
\end{figure}

%%%%%%%%%%%%%%%%%%%%%%%%%%%%%%%%%%%%%%%%%%%%%%%%%%%%%%%%%%%%%%%%%
\subsection{Lower tail}
\label{sec:lowertail}
%%%%%%%%%%%%%%%%%%%%%%%%%%%%%%%%%%%%%%%%%%%%%%%%%%%%%%%%%%%%%%%%%

%We will split the presentation in this subsection into two parts. In the first part we will discuss the key components of the analysis needed to establish the lower bound on the lower tail.  In the second part we will then design a reverse strategy that produces the corresponding upper bound and ultimately makes the lower tail estimates fully exact.
%
%%%%%%%%%%%%%%%%%%%%%%%%%%%%%%%%%%%%%%%%%%%%%%%%%%%%%%%%%%%%%%%%%%
%\subsubsection{Lower bound}
%\label{sec:lowerlowertail}
%%%%%%%%%%%%%%%%%%%%%%%%%%%%%%%%%%%%%%%%%%%%%%%%%%%%%%%%%%%%%%%%%%

In this subsection we focus on the lower bound of the lower tail type of large deviations. The results that we will present below will essentially complement the results from the previous subsection. As earlier, let the elements of $A$ be i.i.d. standard normals. We will below consider the following probability
\begin{eqnarray}
P_{cor} &=& P(\min_{A\w=0,\|\w\|_2\leq 1}\sum_{i=n-k+1}^n \frac{\X_i^T\W_i}{\|\X_i\|_2}+\sum_{i=1}^{n-k}\|\W_{i}\|_2\geq 0)\nonumber \\
&=&
P(\min_{\|\w\|_2\leq 1}\max_{\nu}\sum_{i=n-k+1}^n \frac{\X_i^T\W_i}{\|\X_i\|_2}+\sum_{i=1}^{n-k}\|\W_{i}\|_2+\nu^T A\w\geq 0)\nonumber \\
&=&
P(\max_{\nu}\min_{\|\w\|_2\leq 1}\sum_{i=n-k+1}^n \frac{\X_i^T\W_i}{\|\X_i\|_2}+\sum_{i=1}^{n-k}\|\W_{i}\|_2+\nu^T A\w\geq 0)\nonumber \\
&=&
P(\max_{\nu}\min_{\|\w\|_2= 1}\sum_{i=n-k+1}^n \frac{\X_i^T\W_i}{\|\X_i\|_2}+\sum_{i=1}^{n-k}\|\W_{i}\|_2+\nu^T A\w\geq 0)\nonumber \\
&\leq &
P(\max_{\nu}\min_{\|\w\|_2= 1}\sum_{i=n-k+1}^n \frac{\X_i^T\W_i}{\|\X_i\|_2}+\sum_{i=1}^{n-k}\|\W_{i}\|_2+\nu^T A\w+(g-t_1)\|\nu\|_2\geq 0)/P(g\geq t_1),\nonumber \\
\label{eq:ldpproblower}
\end{eqnarray}
where $P_{cor}=1-P_{err}$ is the probability that (\ref{eq:l2l1}) does produce the solution of (\ref{eq:system}), $t_1$ is a scalar that will be discussed later, and $g$ is a standard normal random variable independent of $A$ (it is probably clear, but for the completeness, we recall that in order to lower bound the lower tail of $P_{err}$, it is enough to upper-bound the complementary quantity $P_{cor}$). We will then make use of the following lemma from \cite{Gordon85}.
\begin{lemma}\cite{Gordon85}
Let $A$ be an $M\times N$ matrix with i.i.d. standard normal components. Let $\g$ and $\h$ be $M\times 1$ and $N\times 1$ vectors, respectively, with i.i.d. standard normal components. Also, let $g$ be a standard normal random variable. Then
\begin{equation}
P(\max_{\nu\in R^n\setminus 0}\min_{\|\w\|_2=1}(\nu^T A\w +\|\nu\|_2 g-\zeta_{\w,\nu,t_1,\X})\geq 0)\leq P(\max_{\nu\in R^n\setminus 0}\min_{\|\w\|_2=1}( \g^T\nu_i+\|\nu\|_2\h^T\w-\zeta_{\w,\nu,t_1,\X})\geq 0).\label{eq:problemmalower}
\end{equation}\label{lemma:unsignedlemmalower}
\end{lemma}
A combination of (\ref{eq:ldpproblower}) and (\ref{eq:problemmalower}) (with a proper adjustment of $\zeta_{\w,\nu,t_1,\X}$) then gives
\begin{eqnarray}
P_{cor}
&\leq &
P(\max_{\nu}\min_{\|\w\|_2= 1}\sum_{i=n-k+1}^n \frac{\X_i^T\W_i}{\|\X_i\|_2}+\sum_{i=1}^{n-k}\|\W_{i}\|_2+\nu^T A\w+(g-t_1)\|\nu\|_2\geq 0)/P(g\geq t_1)\nonumber \\
&\leq &
P(\max_{\lambda}\min_{\|\w\|_2= 1}\lambda(\sum_{i=n-k+1}^n \frac{\X_i^T\W_i}{\|\X_i\|_2}+\sum_{i=1}^{n-k}\|\W_{i}\|_2)+\|\g\|_2 +\h^T\w-t_1\geq 0)/P(g\geq t_1)\nonumber \\
&\leq &
P(\|\g\|_2 -w(\h,S_w')-t_1\geq 0)/P(g\geq t_1)\nonumber \\
&= &
P(\|\g\|_2 -w(\h,S_w)-t_1\geq 0)/P(g\geq t_1),\nonumber \\
\label{eq:ldpprob2lower}
\end{eqnarray}
where we set $\lambda=1/\|\nu\|_2$ and used the considerations from the previous subsection that relate to $w(\h,S_w')$ and $w(\h,S_w)$. Using the Chernoff bound we further have
\begin{equation}
P_{cor}
\leq
P(\|\g\|_2 -w(\h,S_w)-t_1\geq 0)/P(g\geq t_1)
\leq \min_{c_3\geq 0}
Ee^{c_3\|\g\|_2}Ee^{-c_3w(\h,S_w)}e^{-c_3t_1}/P(g\geq t_1).
\label{eq:ldpprob3lower}
\end{equation}
Using (\ref{eq:ldpprob3lower}) one can then establish the following lower tail analogous version of Theorem \ref{thm:ldp1} which provides an upper bound on $P_{cor}$.
\begin{theorem}
Let $A$ be an $M\times N$ matrix in (\ref{eq:system})
with i.i.d. standard normal components. Let
the unknown $\x$ in (\ref{eq:system}) be $k$-block-sparse with the length of its blocks $d$. Further, let the location and the directions of nonzero blocks of $\x$ be arbitrarily chosen but fixed. Let $P_{cor}$ be the probability that the solution of (\ref{eq:l2l1}) is the $k$-block-sparse solution of (\ref{eq:system}). Then
\begin{equation}
P_{cor}\leq \min_{t_1}\min_{c_3\geq 0} Ee^{c_3\|\g\|_2}Ee^{-c_3w(\h,S_w')}e^{-c_3t_1}/P(g\geq t_1).
\label{eq:ldpthm1perrub1lower}
\end{equation}
\label{thm:ldp1lower}
\end{theorem}
We again emphasize that the above bound is valid for any integers $M$, $k$, $n$, and $d$ (of course, as earlier, assuming $dk\leq M\leq dn$ so that the results make sense). As in the previous section, in this section we are also interested in the asymptotic scenario, namely the scenario from Theorem \ref{thm:thmgenweak}. Differently from the previous subsection, here we are interested in the scenarios where $\alpha=\frac{m}{n}\geq (1-\epsilon_{\alpha})\alpha_w$ with $\epsilon_{\alpha}>0$ and $\alpha_w$ being the minimal $\alpha$ such that for a fixed $\beta_w$ the equations of Theorem \ref{thm:thmgenweak} are satisfied. Similarly to what we had earlier for the exponent of $P_{err}$ we now have
\begin{equation}\label{ldpasymp1lower}
  I_{cor}=\lim_{n\rightarrow\infty}\frac{\log{P_{cor}}}{n}.
\end{equation}
Based on Theorem \ref{thm:ldp1} and similarly as in Theorem \ref{thm:ldp2} we have the following LDP type of result.
\begin{theorem}
Assume the setup of Theorem \ref{thm:ldp1lower}. Further, let integers $M$, $k$, and $n$ be large ($dk\leq M\leq dn$). Also, assume the following scaling: $c_3=c_3^{(s)}\sqrt{n}$ and $\gamma=\gamma^{(s)}\sqrt{n}$. Then
\begin{equation}
I_{cor}=\lim_{n\rightarrow\infty}\frac{\log{P_{cor}}}{n}
\leq \min_{c_3^{(s)}\geq 0}\left (-\frac{(c_3^{(s)})^2}{2}+I_{sph}^++\max_{\lambda\geq 0,\gamma^{(s)}\geq 0} ((1-\beta_w)\log{w_1}+\beta_w\log{w_2}+\beta_w\log{w_3}-c_3^{(s)}\gamma^{(s)})\right ),
\label{eq:ldpthm2Ierrub1}
\end{equation}
where
\begin{eqnarray}
I_{sph}^+ &=& \widehat{\gamma_+^{(s)}}c_3^{(s)}-\frac{\alpha d}{2}\log\left (1-\frac{c_3^{(s)}}{2\widehat{\gamma_+^{(s)}}}\right )\nonumber \\
  \widehat{\gamma_+^{(s)}} &=& \frac{2c_3^{(s)}+\sqrt{4(c_3^{(s)})^2+16\alpha d}}{8}\nonumber \\
  w_1 &=& \int_{\hw_i\geq 0}\frac{2^{-d/2}}{\Gamma(d/2)}\hw_i^{d/2-1}e^{-\hw_i/2}e^{-c_3^{(s)}\max(\sqrt{\hw_i}-\lambda,0)^2/4/\gamma^{(s)}}d\hw_i\nonumber \\
  w_2 &=& \frac{1}{\sqrt{2\pi}}\int_{\hw_i}e^{-\hw_i^2/2}e^{-c_3^{(s)}(\hw_i+\lambda)^2/4/\gamma^{(s)}}d\hw_i=e^{(-c_3^{(s)}/4/\gamma^{(s)}\lambda^2)/(1+2c_3^{(s)}/4/\gamma^{(s)})}/\sqrt{1+2c_3^{(s)}/4/\gamma^{(s)}}\nonumber \\
  w_3 &=& \int_{\hw_i\geq 0}\frac{2^{-d/2}}{\Gamma(d/2)}\hw_i^{d/2-1}e^{-\hw_i/2}e^{-c_3^{(s)}\hw_i/4/\gamma^{(s)}}d\hw_i=\left (1/\sqrt{1+2c_3^{(s)}/4/\gamma^{(s)}}\right )^{d-1}.\label{eq:ldpthm2perrub2lower}
\end{eqnarray}\label{thm:ldp2lower}
\end{theorem}
\begin{proof} Follows from Theorem \ref{thm:ldp1lower} optimizing over $t_1$ and by noting that in \cite{StojnicMoreSophHopBnds10} we established
\begin{equation}
I_{sph}^+=\lim_{n\rightarrow\infty}\frac{1}{n}\log(Ee^{c_3^{(s)}\sqrt{n}\|\g\|_2})
=\widehat{\gamma_+^{(s)}}c_3^{(s)}-\frac{\alpha d}{2}\log\left (1-\frac{c_3^{(s)}}{2\widehat{\gamma_+^{(s)}}}\right ),\label{eq:gamaiden2liftlowe}
\end{equation}
where
\begin{equation}
\widehat{\gamma_+^{(s)}}=\frac{2c_3^{(s)}+\sqrt{4(c_3^{(s)})^2+16\alpha d}}{8},\label{eq:gamaiden3liftlower}
\end{equation}
and that analogously one has
\begin{eqnarray}
\lim_{n\rightarrow \infty}\frac{1}{n}\log(Ee^{-c_3w(\h,S_w')})
&=&
\lim_{n\rightarrow \infty}\frac{1}{n}\log\left (Ee^{-c_3\min_{\lambda\geq0,\gamma\geq 0} \left ( \frac{\sum_{i=1}^{n-k}\max(\hw_i-\lambda,0)^2+\sum_{i=n-k+1}^{n}(\hw_i+\lambda)^2+\sum_{i=n+1}^{n+k}\hw_i^2}{4\gamma}+\gamma\right )}\right )\nonumber \\
&=&
\max_{\lambda\geq 0,\gamma^{(s)}\geq 0} ((1-\beta_w)\log{w_1}+\beta_w\log{w_2}+\beta_w\log{w_3}-c_3^{(s)}\gamma^{(s)}).
\end{eqnarray}
\end{proof}
In Figure \ref{fig:ldpIerrlowerupperubd2d10} we present the estimates of the $P_{cor}$ exponents (i.e. of $I_{cor}$) that one can obtain using Theorem \ref{thm:ldp2lower}. Additionally, in Tables \ref{tab:Ierrldplowerubtab2d2} and \ref{tab:Ierrldplowerubtab2d10} we present some of the values we used for $c_3^{(s)}$ , $\gamma^{(s)}$, and $\lambda$ to compute $I_{cor,l}^{(ub)}$ as the estimate. As in the previous subsection, plots in Figure \ref{fig:ldpIerrlowerupperubd2d10} indicate that increasing block-length may shorten the so-called transition zone.

\begin{table}[h]
\caption{A collection of values for $c_3^{(s)}$ , $\gamma^{(s)}$, $\lambda$, and $I_{cor,l}^{(ub)}$ in Theorem \ref{thm:ldp2lower}; $d=2$, $\beta=\frac{1}{3}$}\vspace{.1in}
\hspace{-0in}\centering
\begin{tabular}{||c||c|c|c|c|c||}\hline\hline
$\alpha$          & $ 0.5700 $ & $ 0.5300 $ & $ 0.4900 $ & $ 0.4500 $ & $ 0.4100 $\\ \hline\hline
$c_3^{(s)}$       & $ 0.6198 $ & $ 1.0547 $ & $ 1.5715 $ & $ 2.3026 $ & $ 3.5891 $\\ \hline
$\gamma^{(s)}$    & $ 0.4005 $ & $ 0.3128 $ & $ 0.2391 $ & $ 0.1703 $ & $ 0.1078 $\\ \hline
$\lambda$         & $ 1.0187 $ & $ 1.0810 $ & $ 1.1474 $ & $ 1.2214 $ & $ 1.3014 $\\ \hline\hline
$I_{cor,l}^{(ub)}$& $ \bl{\mathbf{-0.0185}} $ & $ \bl{\mathbf{-0.0495}} $ & $ \bl{\mathbf{-0.0978}} $ & $ \bl{\mathbf{-0.1674}} $ & $ \bl{\mathbf{-0.2649}} $\\ \hline\hline
\end{tabular}
\label{tab:Ierrldplowerubtab2d2}
\end{table}

\begin{table}[h]
\caption{A collection of values for $c_3^{(s)}$ , $\gamma^{(s)}$, $\lambda$, and $I_{cor,l}^{(ub)}$ in Theorem \ref{thm:ldp2lower}; $d=10$, $\beta=\frac{1}{3}$}\vspace{.1in}
\hspace{-0in}\centering
\begin{tabular}{||c||c|c|c|c|c||}\hline\hline
$\alpha$          & $ 0.5500 $ & $ 0.5100 $ & $ 0.4700 $ & $ 0.4300 $ & $ 0.3900 $\\ \hline\hline
$c_3^{(s)}$       & $ 0.4789 $ & $ 1.2466 $ & $ 2.1555 $ & $ 3.4129 $ & $ 5.7283 $\\ \hline
$\gamma^{(s)}$    & $ 1.0590 $ & $ 0.8598 $ & $ 0.6717 $ & $ 0.4896 $ & $ 0.3075 $\\ \hline
$\lambda$         & $ 2.1338 $ & $ 2.2412 $ & $ 2.3487 $ & $ 2.4591 $ & $ 2.5873 $\\ \hline\hline
$I_{cor,l}^{(ub)}$& $ \bl{\mathbf{-0.0132}} $ & $ \bl{\mathbf{-0.0872}} $ & $ \bl{\mathbf{-0.2361}} $ & $ \bl{\mathbf{-0.4791}} $ & $ \bl{\mathbf{-0.8549}} $\\ \hline\hline
\end{tabular}
\label{tab:Ierrldplowerubtab2d10}
\end{table}

\begin{figure}[htb]
%\begin{minipage}[b]{.5\linewidth}
\centering
\centerline{\epsfig{figure=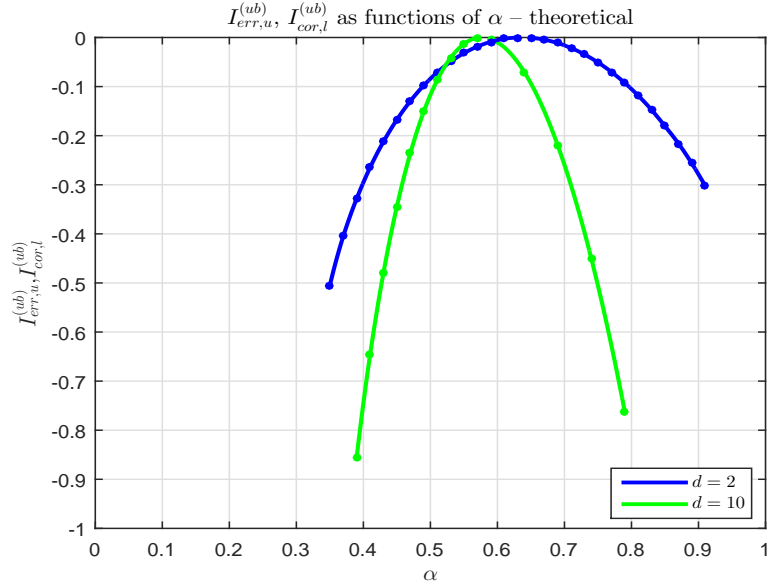,width=11.5cm,height=8cm}}
%\end{minipage}
%\begin{minipage}[b]{.5\linewidth}
%\centering
%\centerline{\epsfig{figure=finprerral08.eps,width=9cm,height=6.5cm}}
%\end{minipage}
\caption{$I_{err,u}^{(ub)}$, $I_{cor,l}^{(ub)}$ as functions of $\alpha$; block-lengths $d=2$ and $d=10$}
\label{fig:ldpIerrlowerupperubd2d10}
\end{figure}

\section{Finite dimensions}
\label{sec:findim}
%%%%%%%%%%%%%%%%%%%%%%%%%%%%%%%%%%%%%%%%%%%%%%%%%%%%%%%%%%%%%%%%%

In this section we will raise the level of considerations a bit (in fact quite a bit). Instead of basically neglecting the true systems dimensions (as was done in the previous sections) we will here specifically focus on them. When it comes to handling problems similar to those considered here, the theories that we have built over the years are fairly superior and we found hardly any other known mathematical tool that can outmatch them. The main features that essentially make them very hard to beat are their exactness, universality, and ultimately their simplicity. A bit lost among them is one additional feature that in a way distinguishes them. Namely, almost all of the results that we created are fairly easy to adjust to finite dimension scenarios. In such scenarios the original simplicity and universality/applicability are preserved. For the problem of interest in this paper, we will towards the end of this section demonstrate that this is indeed true. Before doing that we will focus on another path that we view as essentially the only counterpart that sometimes can match the performance characterization capabilities of the theories that we have built.

%%%%%%%%%%%%%%%%%%%%%%%%%%%%%%%%%%%%%%%%%%%%%%%%%%%%%%%%%%%%%%%%%
\subsection{High-dimensional geometry}
\label{sec:highdimgeom}
%%%%%%%%%%%%%%%%%%%%%%%%%%%%%%%%%%%%%%%%%%%%%%%%%%%%%%%%%%%%%%%%%

The path that we will consider below can be viewed as geometrical and as such is in a rapid contrast with the techniques that we employed in \cite{StojnicCSetam09,StojnicUpper10,StojnicUpperBlock10,StojnicCSetamBlock09} and ultimately in earlier sections of this paper. We will start things off by introducing a few terms and notions well known in high-dimensional geometry, again assuming a high degree of familiarity with some of them. To that end, let $S_{st}$ be the intersection of a convex cone $C_{st}$ and the $n_{st}$-dimensional unit sphere, $S^{n_{st}-1}$, i.e.
\begin{equation}\label{eq:findimSst}
  S_{st}=C_{st}\cap S^{n_{st}-1}.
\end{equation}
Let $\zeta_{\epsilon_{st}}(S_{st})$ be the $\epsilon_{st}$-neighborhood ($0<\epsilon_{st}<\pi/2$) spherical extension of $S_{st}$, i.e.
\begin{equation}\label{eq:findimSsteps}
  \zeta_{\epsilon_{st}}(S_{st})=\{\y\in S^{n_{st}-1}|\quad 0<\min_{\w\in S_{st}}\arccos(\y^T\w)\leq \epsilon_{st} \}.
\end{equation}
Then the following is the so-called spherical Steiner formula (see, e.g. \cite{SW08})
\begin{equation}\label{eq:findimsphSteiner}
  \sigma_{n_{st}-1}(\zeta_{\epsilon_{st}}(S_{st}))=\sum_{i_{st}=0}^{n_{st}-2}g_{n_{st},i_{st}}v_{i_{st}}(S_{st}),
\end{equation}
where
\begin{equation}\label{eq:findimgSteiner}
  g_{n_{st},i_{st}}=\sigma_{i_{st}}(S^{i_{st}})\sigma_{n_{st}-i_[st]-2}(S^{n_{st}-i_{st}-2})
  \int_{0}^{\epsilon_{st}}\cos^{i_{st}}\phi \sin^{n_{st}-i_{st}-2}\phi d\phi,
\end{equation}
and $\sigma_{n_{st}-1}(\cdot)$ is the standard Lebesque measure on the unit $S^{n_{st}-1}$. In fact,
\begin{equation}\label{eq:findimsphmeas}
  \sigma_{n_{st}-1}(S^{n_{st}-1})=\frac{2\pi^{n_{st}/2}}{\Gamma(\frac{n_{st}}{2})}.
\end{equation}
The numbers $v_{i_{st}}(S_{st}),0\leq i_{st}\leq n_{st}-2$, in the spherical Steiner formula are often supplemented by
\begin{equation}\label{eq:findimvm}
  v_{n_{st}-1}(S_{st})=\frac{\sigma_{n_{st}-1}(S_{st})}{\sigma_{n_{st}-1}(S^{n_{st}-1})}.
\end{equation}
One also assumes $v_{i_{st}}(\emptyset)=0$ and $v_{-1}(S_{st})=v_{n_{st}-1}({\cal D}(S_{st}))$, where
\begin{equation}\label{eq:dualSst}
  {\cal D}(S_{st})=\{\y\in S^{n_{st}-1}|\quad \y^T\w\leq 0 \quad \forall \w\in S_{st}\}.
\end{equation}
The numbers $v_{i_{st}}$ are uniquely determined by the Steiner spherical formula and are called the spherical intrinsic volumes. The following spherical Crofton like formula connects the spherical intrinsic volumes and the likelihoods that a random $(n_{st}-j)$-dimensional $G_{n_{st},n_{st}-j}$ linear subspace of $R^{n_{st}}$ from the corresponding Grassmannian misses sets of $S_{st}$ type (see, e.g. \cite{SW08})
\begin{equation}\label{eq:findimCrft}
P(G_{n_{st},n_{st}-j}\cap S_{st}\neq\emptyset)=2\sum_{k_{st}=0}^{\lfloor\frac{n_{st}-1-j}{2}\rfloor}v_{j+2k_{st}+1}(S_{st}).
\end{equation}
Choosing $S_{st}=S_{w}'$ and connecting (\ref{eq:ldpprob}) and (\ref{eq:findimCrft}) while keeping in mind a rotational invariance of $A$ (alternatively one can instead of Gaussian $A$'s consider $A$'s that have the null-space distributed uniformly randomly in the corresponding Grassmaninans) one has
\begin{equation}\label{eq:findimperrcrft}
P_{err}= P(\min_{\w\in S_w'}\|A\w\|_2\leq 0)=P(G_{dn,dn-M}\cap S_{w}'\neq\emptyset)=2\sum_{k_{st}=0}^{\lfloor\frac{dn-1-M}{2}\rfloor}v_{M+2k_{st}+1}(S_{w}').
\end{equation}
The above formula then positions Crofton like (\ref{eq:findimCrft}) as a potentially powerful tool in determining $P_{err}$. In fact, the combination of (\ref{eq:findimsphSteiner}) and (\ref{eq:findimCrft}) is indeed a well known principal tool in high-dimensional integral geometry. The problem is that it typically remains of not much practical use as it requires addressing several major problems. Resolving these problems is in general much harder than deriving any of (\ref{eq:findimsphSteiner}) and (\ref{eq:findimCrft}) (we should though emphasize that while deriving (\ref{eq:findimsphSteiner}) is basically trivial for high-dimensional integral geometry experts, deriving (\ref{eq:findimCrft}) is not super obvious and typically uses as key steps the ideas that date back to the works of Santalo, Hadwiger, and McMullen, obviously the integral geometry greats). Namely, to use (\ref{eq:findimsphSteiner}) one should be able first to determine $\sigma_{n_{st}-1}(\zeta_{\epsilon_{st}}(S_{st}))$ and then to hope that the system on the right side of (\ref{eq:findimsphSteiner}) is (computationally speaking) well-posed. In the remaining parts of this section we will try to address these problems.

In general, determining $\sigma_{n_{st}-1}(\zeta_{\epsilon_{st}}(S_{st}))$ analytically is rarely known to be possible. Below, we will present a strategy that heavily relies on the machinery that we introduced in \cite{StojnicBlockDepNon10,StojnicCSetamBlock09,StojnicISIT2010binary,StojnicICASSP10knownsupp} and to a degree utilized in the previous sections of this paper. Using such a strategy we will be able to determine asymptotically optimal estimates of $\sigma_{dn-1}(\zeta_{\epsilon_{st}}(S_{w}'))$ that will also turn out to be quite close approximations in finite dimensions. Moreover, the approximations that we will derive will in fact be rigorous upper bounds on $\sigma_{dn-1}(\zeta_{\epsilon_{st}}(S_{w}'))$ (the numerical results that we will present will in fact confirm that these bounds are indeed very, very close to the exact values).

%%%%%%%%%%%%%%%%%%%%%%%%%%%%%%%%%%%%%%%%%%%%%%%%%%%%%%%%%%%%%%%%%
\subsubsection{Determining $\frac{\sigma_{dn-1}(\zeta_{\epsilon_{st}}(S_{w}'))}{\sigma_{dn-1}(S^{dn-1})}$}
\label{sec:sigmast}
%%%%%%%%%%%%%%%%%%%%%%%%%%%%%%%%%%%%%%%%%%%%%%%%%%%%%%%%%%%%%%%%%

We will find it a bit more natural to work with a scaled version of $\sigma_{dn-1}(\zeta_{\epsilon_{st}}(S_{w}'))$, namely the ratio $\frac{\sigma_{dn-1}(\zeta_{\epsilon_{st}}(S_{w}'))}{\sigma_{dn-1}(S^{dn-1})}$ (of course, determining $\frac{\sigma_{dn-1}(\zeta_{\epsilon_{st}}(S_{w}'))}{\sigma_{dn-1}(S^{dn-1})}$ automatically determines $\sigma_{dn-1}(\zeta_{\epsilon_{st}}(S_{w}'))$).
We start by quickly modifying the $\zeta_{\epsilon_{st}}(S_{w}')$ given in (\ref{eq:findimSsteps}) with $S_{st}$ obviously replaced by $S_{w}'$. To that end we have
\begin{eqnarray}\label{eq:findimSstepsmod}
    \zeta_{\epsilon_{st}}(S_{w}') &=& \{\y\in S^{dn-1}|\quad 0<\min_{\w\in S_{w}'}\arccos(\y^T\w)\leq \epsilon_{st} \}\nonumber \\
   &=& \{\y\in S^{dn-1}|\quad 1>\cos (\min_{\w\in S_{w}'}\arccos(\y^T\w))\geq \cos(\epsilon_{st}) \}\nonumber\\
   &=& \{\y\in S^{dn-1}|\quad 1>\max_{\w\in S_{w}'}\cos(\arccos(\y^T\w))\geq \cos(\epsilon_{st}) \}\nonumber \\
   &=& \{\y\in S^{dn-1}|\quad \max_{\w\in S_{w}'}(\y^T\w)\geq \cos(\epsilon_{st}) \}.
\end{eqnarray}
We then impose the uniform measure over the unit sphere $S^{dn-1}$ through the standard Gaussians. Namely, let, as earlier, $\h$ be an $dn$-dimensional vector of i.i.d. standard normals. Then, due to the rotational invariance of $\h$, we have
\begin{equation}\label{eq:sigmast1}
  \frac{\sigma_{dn-1}(\zeta_{\epsilon_{st}}(S_{w}'))}{\sigma_{dn-1}(S^{dn-1})}=\frac{\sigma_{dn-1}(\{\y\in S^{dn-1}|\quad \max_{\w\in S_{w}'}(\y^T\w)\geq \cos(\epsilon_{st}) \})}{\sigma_{dn-1}(S^{dn-1})}
  =P(\max_{\w\in S_{w}'}\frac{\h^T\w}{\|\h\|_2}\geq \cos(\epsilon_{st})).
\end{equation}
Based on (\ref{eq:widthdefSwpr}) and  (\ref{eq:sigmast1}) we easily obtain
\begin{equation}\label{eq:sigmast2}
\frac{\sigma_{dn-1}(\zeta_{\epsilon_{st}}(S_{w}'))}{\sigma_{dn-1}(S^{dn-1})}  =P((\max_{\w\in S_{w}'}(\h^T\w))^2\geq \cos^2(\epsilon_{st})\|\h\|_2^2)=P((w(\h,S_w'))^2\geq \cos^2(\epsilon_{st})\|\h\|_2^2).
\end{equation}
Similarly to what was done in (\ref{eq:ldpprob4}) and ultimately based on the rotation considerations from \cite{StojnicCSetamBlock09} we also have
\begin{equation}\label{eq:sigmast3}
\frac{\sigma_{dn-1}(\zeta_{\epsilon_{st}}(S_{w}'))}{\sigma_{dn-1}(S^{dn-1})}  =P((w(\h,S_w'))^2\geq \cos^2(\epsilon_{st})\|\h\|_2^2)=P((w(\h,S_w))^2\geq \cos^2(\epsilon_{st})\|\h\|_2^2).
\end{equation}
Now, set
\begin{eqnarray}
w_\lambda(\h,\Sw) &=&\sqrt{\sum_{i=1}^{n-k}\max(\hw_i-\lambda,0)^2+\sum_{i=n-k+1}^{n}(\hw_i+\lambda)^2+\sum_{i=n+1}^{n+k}\hw_i^2}\nonumber \\
w_{\lambda,1} &=& \sum_{i=1}^{n-k}\max(\hw_i-\lambda,0)^2\nonumber \\
w_{\lambda,2} &=& \sum_{i=n-k+1}^{n}(\hw_i+\lambda)^2\nonumber \\
w_{\lambda,3} &=& \sum_{i=n+1}^{n+k}\hw_i^2.
\label{eq:sigmastldpwhSwlam}
\end{eqnarray}
Then a combination of (\ref{eq:sigmast3}) and (\ref{eq:ldpwhSw}) further gives
\begin{eqnarray}\label{eq:sigmast4}
\frac{\sigma_{dn-1}(\zeta_{\epsilon_{st}}(S_{w}'))}{\sigma_{dn-1}(S^{dn-1})}  &=& P((w(\h,S_w))^2\geq \cos^2(\epsilon_{st})\|\h\|_2^2)\nonumber \\
  &=& P((\min_{\lambda\geq0}w_\lambda(\h,S_w))^2\geq \cos^2(\epsilon_{st})\|\h\|_2^2)\nonumber \\
  &\leq& \min_{\lambda\geq0} P((w_\lambda(\h,S_w))^2\geq \cos^2(\epsilon_{st})\|\h\|_2^2)\nonumber \\
  &=& 1-\max_{\lambda\geq0} P((w_\lambda(\h,S_w))^2< \cos^2(\epsilon_{st})\|\h\|_2^2).
\end{eqnarray}
We will now focus on computing $P((w_\lambda(\h,S_w))^2< \cos^2(\epsilon_{st})\|\h\|_2^2)$. We start by observing the following simple identity
\begin{equation}\label{eq:sigmast5}
P_\lambda=P((w_\lambda(\h,S_w))^2< \cos^2(\epsilon_{st})\|\h\|_2^2)=
 \frac{1}{\sqrt{2\pi}^{dn}}\int_{\h}h((-w_\lambda(\h,S_w))^2+\cos^2(\epsilon_{st})\|\h\|_2^2)e^{-\frac{\|\h\|^2}{2}}d\h,
\end{equation}
where $h(t)$ is the standard indicator step function. The following characterization of the standard step function we found fairly useful
\begin{equation}\label{eq:sigmasthfun}
  h(t)=\lim_{\epsilon_{wf}\rightarrow 0_+}\int_{-\infty}^{\infty}\frac{e^{jw_ft}}{2\pi j(w_f-j\epsilon_{wf})}dw_f.
\end{equation}
Combining (\ref{eq:sigmast5}) and (\ref{eq:sigmasthfun}) we then have
\begin{eqnarray}\label{eq:sigmast6}
P_\lambda &=&  \frac{1}{\sqrt{2\pi}^{dn}}\int_{\h}h((-w_\lambda(\h,S_w))^2+\cos^2(\epsilon_{st})\|\h\|_2^2\geq 0)e^{-\frac{\|\h\|_2^2}{2}}d\h\nonumber\\
   &=&  \frac{1}{\sqrt{2\pi}^{dn}}\int_{\h}
   \lim_{\epsilon_{wf}\rightarrow 0_+}\int_{-\infty}^{\infty}\frac{e^{-jw_f((w_\lambda(\h,S_w))^2-\cos^2(\epsilon_{st})\|\h\|_2^2)}}{2\pi j(w_f-j\epsilon_{wf})}dw_f
   e^{-\frac{\|\h\|_2^2}{2}}d\h \nonumber\\
   &=&
   \lim_{\epsilon_{wf}\rightarrow 0_+}\int_{-\infty}^{\infty}\frac{1}{{2\pi j(w_f-j\epsilon_{wf})}}\frac{1}{\sqrt{2\pi}^{dn}}\int_{\h}e^{-jw_f(w_{\lambda,1}+w_{\lambda,2}+w_{\lambda,3}-\cos^2(\epsilon_{st})\|\h\|_2^2)}
   e^{-\frac{\|\h\|_2^2}{2}}d\h dw_f \nonumber\\
   &=&
   \lim_{\epsilon_{wf}\rightarrow 0_+}\int_{-\infty}^{\infty}\frac{1}{{2\pi j(w_f-j\epsilon_{wf})}}I_{\lambda,1}(w_f,\epsilon_{st})I_{\lambda,2}(w_f,\epsilon_{st})I_{\lambda,3}(w_f,\epsilon_{st})dw_f,
\end{eqnarray}
where
\begin{eqnarray}\label{eq:sigmast7}
  I_{\lambda,1}(w_f,\epsilon_{st}) &=& \frac{1}{\sqrt{2\pi}^{d(n-k)}}\int e^{-jw_f(w_{\lambda,1}-\cos^2(\epsilon_{st})\|\h_{1:d(n-k)}\|_2^2)}
   e^{-\frac{\|\h_{1:d(n-k)}\|_2^2}{2}}d\h_{1:d(n-k)} \nonumber\\
  I_{\lambda,2}(w_f,\epsilon_{st}) &=& \frac{1}{\sqrt{2\pi}^{k}}\int e^{-jw_f(w_{\lambda,2}-\cos^2(\epsilon_{st})
 \sum_{i=n-k+1}^{n}\|\h_{d(i-1)+1}\|_2^2)}
   \prod_{i=n-k+1}^{n}e^{-\frac{\|\h_{d(i-1)+1}\|_2^2}{2}}d\h_{d(i-1)+1} \nonumber\\
  I_{\lambda,3}(w_f,\epsilon_{st}) &=& \frac{1}{\sqrt{2\pi}^{(d-1)(k)}}\int e^{-jw_f(w_{\lambda,3}-\cos^2(\epsilon_{st})
  \sum_{i=n-k+1}^{n}\|\h_{d(i-1)+2:di}\|_2^2)}
   \prod_{i=n-k+1}^{n}e^{-\frac{\|\h_{d(i-1)+2:di}\|_2^2}{2}}d\h_{d(i-1)+2:di}.\nonumber\\
\end{eqnarray}
We now analyze separately each of $I_{\lambda,1}(w_f,\epsilon_{st})$, $I_{\lambda,2}(w_f,\epsilon_{st})$, and $I_{\lambda,3}(w_f,\epsilon_{st})$.

\underline{1) \textbf{\emph{Computing $I_{\lambda,1}(w_f,\epsilon_{st})$}}}

From (\ref{eq:workww00}) we recall that $\H_i = (\h_{(i-1)d + 1}, \h_{(i-1)d + 2}, \ldots, \h_{id})^T, i = 1, 2, \ldots, n-k$, and from (\ref{eq:defhweak}) we recall that $\hw_i=\|\H_i\|_2, i = 1, 2, \ldots, n-k$. Then we have
\begin{eqnarray}\label{eq:sigmastIwf1}
  I_{\lambda,1}(w_f,\epsilon_{st}) &=& \frac{1}{\sqrt{2\pi}^{d(n-k)}}\int e^{-jw_f(w_{\lambda,1}-\cos^2(\epsilon_{st})\|\h_{1:d(n-k)}\|_2^2)}
   e^{-\frac{\|\h_{1:d(n-k)}\|_2^2}{2}}d\h_{1:d(n-k)} \nonumber\\
&=& \frac{1}{\sqrt{2\pi}^{d(n-k)}}\int e^{-jw_f(\sum_{i=1}^{n-k}\max(\hw_i-\lambda,0)^2-\cos^2(\epsilon_{st})\|\h_{1:d(n-k)}\|_2^2)}
   e^{-\frac{\|\h_{1:d(n-k)}\|_2^2}{2}}d\h_{1:d(n-k)} \nonumber\\
&=& \frac{1}{\sqrt{2\pi}^{d(n-k)}}\int e^{-jw_f(\sum_{i=1}^{n-k}\max(\hw_i-\lambda,0)^2-\cos^2(\epsilon_{st})\sum_{i=1}^{n-k}\hw_i^2)}
   e^{-\frac{\sum_{i=1}^{n-k}\hw_i^2}{2}}d\h_{1:d(n-k)} \nonumber\\
&=& \frac{1}{\sqrt{2\pi}^{d(n-k)}}\int\prod_{i=1}^{n-k}e^{-jw_f(\max(\hw_i-\lambda,0)^2-\cos^2(\epsilon_{st})\hw_i^2)}
   e^{-\frac{\hw_i^2}{2}}d\h_{(i-1)d+1:id} \nonumber\\
&=& \prod_{i=1}^{n-k}\frac{1}{\sqrt{2\pi}^{d}}\int e^{-jw_f(\max(\hw_i-\lambda,0)^2-\cos^2(\epsilon_{st})\hw_i^2)}
   e^{-\frac{\hw_i^2}{2}}d\h_{(i-1)d+1:id} \nonumber\\
&=& \left (\frac{1}{\sqrt{2\pi}^{d}}\int e^{-jw_f(\max(\hw_i-\lambda,0)^2-\cos^2(\epsilon_{st})\hw_i^2)}
   e^{-\frac{\hw_i^2}{2}}d\h_{(i-1)d+1:id}\right )^{n-k},1\leq i\leq n-k.
\end{eqnarray}
Using the distribution function for $\hw_i,1\leq i\leq n-k$, from (\ref{eq:ldpthm1perrub3}) we further have
\begin{eqnarray}\label{eq:sigmastIwf10}
  I_{\lambda,1}(w_f,\epsilon_{st})
&=& \left (\frac{1}{\sqrt{2\pi}^{d}}\int e^{-jw_f(\max(\hw_i-\lambda,0)^2-\cos^2(\epsilon_{st})\hw_i^2)}
   e^{-\frac{\hw_i^2}{2}}d\h_{(i-1)d+1:id}\right )^{n-k}\nonumber \\
&=& \left (\int_{\hw_i\geq 0}e^{-jw_f(\max(\sqrt{\hw_i}-\lambda,0)^2-\cos^2(\epsilon_{st})\hw_i)}
   \frac{2^{-d/2}}{\Gamma(d/2)}\hw_i^{d/2-1}e^{-\hw_i/2}d\hw_i\right )^{n-k}.
\end{eqnarray}

\underline{2) \textbf{\emph{Computing $I_{\lambda,2}(w_f,\epsilon_{st})$}}}

From (\ref{eq:defhweak}) we recall that $\hw_{n-k+1:n}=(-\h_{(n-k)d+1},-\h_{(n-k+1)d+1},\dots,-\h_{(n-1)d+1},)$, i.e. $\hw_i =-\h_{(i-1)d+1},i=n-k+1,n-k+2, \ldots, n$. Then we have
\begin{eqnarray}\label{eq:sigmastIwf2}
  I_{\lambda,2}(w_f,\epsilon_{st}) &=& \frac{1}{\sqrt{2\pi}^{k}}\int e^{-jw_f(w_{\lambda,2}-\cos^2(\epsilon_{st})
 \sum_{i=n-k+1}^{n}\h_{d(i-1)+1}^2)}
   \prod_{i=n-k+1}^{n}e^{-\frac{\h_{d(i-1)+1}^2}{2}}d\h_{d(i-1)+1} \nonumber\\
&=& \frac{1}{\sqrt{2\pi}^{k}}\int e^{-jw_f(\sum_{i=n-k+1}^{n}(\hw_i+\lambda)^2-\cos^2(\epsilon_{st})
 \sum_{i=n-k+1}^{n}\h_{d(i-1)+1}^2)}
   \prod_{i=n-k+1}^{n}e^{-\frac{\h_{d(i-1)+1}^2}{2}}d\h_{d(i-1)+1} \nonumber\\
&=& \frac{1}{\sqrt{2\pi}^{k}}\int \prod_{i=n-k+1}^{n} e^{-jw_f((-\h_{d(i-1)+1}+\lambda)^2-\cos^2(\epsilon_{st})
 \h_{d(i-1)+1}^2)}
   e^{-\frac{\h_{d(i-1)+1}^2}{2}}d\h_{d(i-1)+1} \nonumber\\
&=& \left ( \frac{1}{\sqrt{2\pi}}\int e^{-jw_f((-\h_{d(i-1)+1}+\lambda)^2-\cos^2(\epsilon_{st})
 \h_{d(i-1)+1}^2)}
   e^{-\frac{\h_{d(i-1)+1}^2}{2}}d\h_{d(i-1)+1}\right )^{k} \nonumber\\
   &=& \left ( exp\left (-jw_f\lambda^2+\left (2\lambda jw_f/\sqrt{1-2j w_f(\cos^2(\epsilon_{st})-1)}\right )^2/2\right )/\sqrt{1-2j w_f(\cos^2(\epsilon_{st})-1)}\right )^k.\nonumber \\
\end{eqnarray}

% af=sqrt(1-2*i*w*(cos(epcrft)^2-1));
% bf=2*lam*i*w/sqrt(1-2*i*w*(cos(epcrft)^2-1));
% qqupf=exp(-i*w*lam^2+(2*lam*i*w/sqrt(1-2*i*w*(cos(epcrft)^2-1)))^2/2)/sqrt(1-2*i*w*(cos(epcrft)^2-1));

\underline{3) \textbf{\emph{Computing $I_{\lambda,3}(w_f,\epsilon_{st})$}}}

From (\ref{eq:defhweak}) we recall that $\hw_{i+k}=\|\h_{d(i-1)+2:di}\|_2, i =n-k+1,n-k+2,\ldots,n$. Then we have
\begin{eqnarray}\label{eq:sigmastIwf1}
  I_{\lambda,3}(w_f,\epsilon_{st}) &=& \frac{1}{\sqrt{2\pi}^{(d-1)(k)}}\int e^{-jw_f(w_{\lambda,3}-\cos^2(\epsilon_{st})
  \sum_{i=n-k+1}^{n}\|\h_{d(i-1)+2:di}\|_2^2)}
   \prod_{i=n-k+1}^{n}e^{-\frac{\|\h_{d(i-1)+2:di}\|_2^2}{2}}d\h_{d(i-1)+2:di} \nonumber\\
&=& \frac{1}{\sqrt{2\pi}^{(d-1)(k)}}\int e^{-jw_f(\sum_{i=n-k+1}^{n}\hw_{i+k}^2-\cos^2(\epsilon_{st})
  \sum_{i=n-k+1}^{n}\hw_{i+k}^2)}
   \prod_{i=n-k+1}^{n}e^{-\frac{\hw_{i+k}^2}{2}}d\h_{d(i-1)+2:di} \nonumber\\
&=& \frac{1}{\sqrt{2\pi}^{(d-1)(k)}}\int \prod_{i=n-k+1}^{n} e^{-jw_f(\hw_{i+k}^2-\cos^2(\epsilon_{st})
  \hw_{i+k}^2)}
   e^{-\frac{\hw_{i+k}^2}{2}}d\h_{d(i-1)+2:di} \nonumber\\
&=&  \left ( \frac{1}{\sqrt{2\pi}^{(d-1)}}\int e^{-jw_f(\hw_{i+k}^2-\cos^2(\epsilon_{st})
  \hw_{i+k}^2)}
   e^{-\frac{\hw_{i+k}^2}{2}}d\h_{d(i-1)+2:di} \right )^k,n-k+1\leq i\leq n.
\end{eqnarray}
Using the distribution function for $\hw_{i+k},n-k+1\leq i\leq n$, from (\ref{eq:ldpthm1perrub3}) we further have
\begin{eqnarray}\label{eq:sigmastIwf30}
  I_{\lambda,3}(w_f,\epsilon_{st})
&=&  \left ( \frac{1}{\sqrt{2\pi}^{(d-1)}}\int e^{-jw_f(\hw_{i+k}^2-\cos^2(\epsilon_{st})
  \hw_{i+k}^2)}
   e^{-\frac{\hw_{i+k}^2}{2}}d\h_{d(i-1)+2:di} \right )^k \nonumber \\
&=& \left (\int_{\hw_i\geq 0}e^{-jw_f(\hw_{i+k}-\cos^2(\epsilon_{st})
  \hw_{i+k})}
   \frac{2^{-(d-1)/2}}{\Gamma((d-1)/2)}\hw_i^{(d-1)/2-1}e^{-\hw_i/2}d\hw_i\right )^{k}\nonumber \\
&=& \frac{1}{\sqrt{1-2j(w_f(\cos^2(\epsilon_{st})-1))}^{(d-1)k}}.
\end{eqnarray}

We summarize the results from this subsection in the following theorem.
\begin{theorem}
Let $0<\epsilon_{st}< \frac{\pi}{2}$. Set $S_{w}'$ and $\zeta_{\epsilon_{st}}(S_{w}')$ as in (\ref{eq:defSwpr}) and (\ref{eq:findimSstepsmod}), respectively. Let $\sigma_{dn-1}(\cdot)$ be the standard Lebesque measure on the unit $S^{dn-1}$ sphere. Then
\begin{equation}\label{eq:thmsigmast1}
  \frac{\sigma_{dn-1}(\zeta_{\epsilon_{st}}(S_{w}'))}{\sigma_{dn-1}(S^{dn-1})}\leq 1-\max_{\lambda\geq 0}\lim_{\epsilon_{wf}\rightarrow 0_+}\int_{-\infty}^{\infty}\frac{1}{{2\pi j(w_f-j\epsilon_{wf})}}I_{\lambda,1}(w_f,\epsilon_{st})I_{\lambda,2}(w_f,\epsilon_{st})I_{\lambda,3}(w_f,\epsilon_{st})dw_f,
\end{equation}
where $I_{\lambda,1}(w_f,\epsilon_{st})$, $I_{\lambda,2}(w_f,\epsilon_{st})$, and $I_{\lambda,3}(w_f,\epsilon_{st})$ are as given in (\ref{eq:sigmastIwf10}), (\ref{eq:sigmastIwf2}), and (\ref{eq:sigmastIwf30}), respectively.
\label{thm:sigmast1}
\end{theorem}
\begin{proof}
  Follows from the discussion presented above.
\end{proof}
To give an idea how close to the true values are the estimates for $\frac{\sigma_{dn-1}(\zeta_{\epsilon_{st}}(S_{w}'))}{\sigma_{dn-1}(S^{dn-1})}$  obtained through Theorem \ref{thm:sigmast1} we below in Tables \ref{tab:sigmastubtab3d2} and \ref{tab:sigmastubtab3d10} present
the results for two different block lengths, $d=2$ and $d=10$. As the tables indicate, Theorem \ref{thm:sigmast1} gives values that are indeed very close to the exact ones. We also add that in this introductory paper, we chose to present a simple strategy that already works fairly well. One can actually improve on it but that requires a bit of additional technical effort (similar is also true for some of the considerations from the following section). To ensure that here we preserve the elegance of the introductory concepts we will present those considerations in a separate paper.
\begin{table}[h]
\caption{Simulated $\frac{\sigma_{dn-1}(\zeta_{\epsilon_{st}}(S_{w}'))}{\sigma_{dn-1}(S^{dn-1})}$ and its upper bound obtained through Theorem \ref{thm:sigmast1}; $d=2$, $k=6$, $n=18$}\vspace{.1in}
\hspace{-0in}\centering
\begin{tabular}{||c||c|c|c|c|c||}\hline\hline
$\epsilon_{st}$                                      & $ 0.5000 $ & $ 0.6000 $ & $ 0.7000 $ & $ 0.8000 $ & $ 0.9000 $\\ \hline\hline
$\lambda$                                            & $ 0.7152 $ & $ 0.8441 $ & $ 0.9613 $ & $ 1.0858 $ & $ 1.2015 $\\ \hline
$\frac{\sigma_{dn-1}(\zeta_{\epsilon_{st}}(S_{w}'))}{\sigma_{dn-1}(S^{dn-1})}$ -- upper bound& $ \bl{\mathbf{0.1465}} $ & $ \bl{\mathbf{0.3538}} $ & $ \bl{\mathbf{0.6145}} $ & $ \bl{\mathbf{0.8325}} $ & $ \bl{\mathbf{0.9505}} $\\ \hline
$\frac{\sigma_{dn-1}(\zeta_{\epsilon_{st}}(S_{w}'))}{\sigma_{dn-1}(S^{dn-1})}$ -- simulated & $ \prp{\mathbf{0.1388}} $ & $ \prp{\mathbf{0.3410}} $ & $ \prp{\mathbf{0.5980}} $ & $ \prp{\mathbf{0.8261}} $ & $ \prp{\mathbf{0.9471}} $\\ \hline\hline
\end{tabular}
\label{tab:sigmastubtab3d2}
\end{table}
\begin{table}[h]
\caption{Simulated $\frac{\sigma_{dn-1}(\zeta_{\epsilon_{st}}(S_{w}'))}{\sigma_{dn-1}(S^{dn-1})}$ and its upper bound obtained through Theorem \ref{thm:sigmast1}; $d=10$, $k=6$, $n=18$}\vspace{.1in}
\hspace{-0in}\centering
\begin{tabular}{||c||c|c|c|c|c||}\hline\hline
$\epsilon_{st}$                                      & $ 0.6000 $ & $ 0.6500 $ & $ 0.7000 $ & $ 0.7500 $ & $ 0.8000 $\\ \hline\hline
$\lambda$                                            & $ 1.8045 $ & $ 1.9119 $ & $ 2.0291 $ & $ 2.1878 $ & $ 2.2659 $\\ \hline
$\frac{\sigma_{dn-1}(\zeta_{\epsilon_{st}}(S_{w}'))}{\sigma_{dn-1}(S^{dn-1})}$ -- upper bound& $ \bl{\mathbf{0.0619}} $ & $ \bl{\mathbf{0.1877}} $ & $ \bl{\mathbf{0.4391}} $ & $ \bl{\mathbf{0.7254}} $ & $ \bl{\mathbf{0.9122}} $\\ \hline
$\frac{\sigma_{dn-1}(\zeta_{\epsilon_{st}}(S_{w}'))}{\sigma_{dn-1}(S^{dn-1})}$ -- simulated & $ \prp{\mathbf{0.0482}} $ & $ \prp{\mathbf{0.1727}} $ & $ \prp{\mathbf{0.4346}} $ & $ \prp{\mathbf{0.7168}} $ & $ \prp{\mathbf{0.9085}} $\\ \hline\hline
\end{tabular}
\label{tab:sigmastubtab3d10}
\end{table}

%%%%%%%%%%%%%%%%%%%%%%%%%%%%%%%%%%%%%%%%%%%%%%%%%%%%%%%%%%%%%%%%%
\subsubsection{Determining $v_{M+2k_{st}+1}(S_w')$}
\label{sec:vst}
%%%%%%%%%%%%%%%%%%%%%%%%%%%%%%%%%%%%%%%%%%%%%%%%%%%%%%%%%%%%%%%%%

As discussed above, to determine $v_{M+2k_{st}+1}(S_w')$ (and ultimately $P_{err}$ in (\ref{eq:findimperrcrft})) one in principle can utilize the spherical Steiner formula (\ref{eq:findimsphSteiner}). That would of course be possible if one can determine $\sigma_{dn-1}(\zeta_{\epsilon_{st}}(S_{w}'))$ and if the system in (\ref{eq:findimsphSteiner}) can be solved. To solve the system in (\ref{eq:findimsphSteiner}) we employed the following standard $(\lambda_{reg},p_{reg},q_{reg})$ regression
\begin{eqnarray}\label{eq:detvst1}
  \min_{v_{i_{st}}(S_w')\geq 0} & & \|\sigma(S_w')-G\v(S_w')\|_{p_{reg}}+\lambda_{reg}\|\v(S_w')\|_{q_{reg}} \nonumber \\
  \mbox{subjec to} & &  \sigma(S_w')-G\v(S_w')\geq \0\nonumber \\
                   & & \sum_{i_{st}=0}^{dn-2}v_{i_{st}}(S_w')=1-v_{dn-1}(S_w')\nonumber \\
                   & & \sum_{i_{st}=0}^{dn-1}(-1)^{i_{st}}v_{i_{st}}(S_w')=0,
\end{eqnarray}
where
\begin{eqnarray}\label{eq:detvst2}
  \sigma(S_w') &=& (\sigma_{dn-1}(\zeta_{\epsilon_{st}^{(1)}}(S_{w}')),\sigma_{dn-1}(\zeta_{\epsilon_{st}^{(2)}}(S_{w}')),\dots,\sigma_{dn-1}(\zeta_{\epsilon_{st}^{(n_{\epsilon_{st}})}}(S_{w}')))^T \nonumber \\
  \v(S_w') &=& (v_{0}(S_w'),v_{1}(S_w'),\dots,v_{dn-2}(S_w'))^T \nonumber \\
  G_{n_{st},i_{st}} &=& g_{n_{st},i_{st}},0\leq i_{st}\leq nd-2,1\leq n_{st}\leq n_{\epsilon_{st}}\nonumber \\
  v_{dn-1}(S_w') &=& \frac{\sigma_{dn-1}(S_w')}{\sigma_{dn-1}(S^{dn-1})}.
\end{eqnarray}
Following considerations from earlier parts of this section one then has
\begin{equation}\label{eq:detvst3}
  v_{dn-1}(S_w') = \frac{\sigma_{dn-1}(S_w')}{\sigma_{dn-1}(S^{dn-1})}=P(\sum_{i=1}^{n}\hw_i\geq 0).
\end{equation}
Similarly to (\ref{eq:sigmast6}) and (\ref{eq:sigmast7}) we have
\begin{equation}\label{eq:detvst4}
v_{dn-1}(S_w')
   =
   \lim_{\epsilon_{wf}\rightarrow 0_+}\int_{-\infty}^{\infty}\frac{1}{{2\pi j(w_f-j\epsilon_{wf})}}I_{\lambda,1}^v(w_f)I_{\lambda,2}^v(w_f)dw_f,
\end{equation}
where
\begin{eqnarray}\label{eq:detvst5}
  I_{\lambda,1}^v(w_f) &=& \frac{1}{\sqrt{2\pi}^{d(n-k)}}\int e^{jw_f(\|\h_{1:d(n-k)}\|_2)}
   e^{-\frac{\|\h_{1:d(n-k)}\|_2^2}{2}}d\h_{1:d(n-k)}\nonumber \\
   &=& \left (\int_{\hw_i\geq 0}e^{jw_f\sqrt{\hw_i}}
   \frac{2^{-d/2}}{\Gamma(d/2)}\hw_i^{d/2-1}e^{-\hw_i/2}d\hw_i\right )^{n-k} \nonumber\\
  I_{\lambda,2}^v(w_f) &=& \frac{1}{\sqrt{2\pi}^{k}}\int e^{jw_f
 \sum_{i=n-k+1}^{n}\h_{d(i-1)+1}}
   \prod_{i=n-k+1}^{n}e^{-\frac{\h_{d(i-1)+1}^2}{2}}d\h_{d(i-1)+1}=e^{-kw_f^2/2}.\nonumber\\
\end{eqnarray}

In Figures \ref{fig:findimd2d10crft} and \ref{fig:findimd2scaledcrft} we present the results that we obtained based on the above considerations. Clear;y, the figures are designed to showcase the effect that both, increasing the block-length as well as increasing the number of blocks have on the probability of error and its various estimates that we provide. We chose $\epsilon_{st}^{(i)}\in[0.4,1]$ and $\epsilon_{st}^{(i)}=0.4+\Delta_{\epsilon_{st}}i,i\in\{1,2,\dots\}$. For $d=2$, $k=6$, $n=18$ we chose $\Delta_{\epsilon_{st}}=0.01$, for $d=2$, $k=24$, $n=72$ we chose $\Delta_{\epsilon_{st}}=0.004$, and for $d=10$, $k=6$, $n=18$ we chose $\Delta_{\epsilon_{st}}=0.002$. There are several quantities that are plotted in each figure: $p_{sim}$ which is a simulated probability of error; $p_{num}$ which is obtained based on the above considerations with $\sigma_{dn-1}(\zeta_{\epsilon_{st}}(S_{w}'))$ simulated; $p_{apx}$ which is obtained fully analytically based on the theoretical upper bound on $\sigma_{dn-1}(\zeta_{\epsilon_{st}}(S_{w}'))$ derived in the previous section and the above regression mechanism to estimate $\v(S_w')$. Finally we present a fully analytical upper bound on the probability of error $p_{ub}^{(ldp)}$ obtained using results of Theorem \ref{thm:ldp1} and recognizing
that $\|\g\|_2^2$ in Theorem \ref{thm:ldp1} is a $\chi$-square random variable. As can be seen from figures, $p_{sim}$ and $p_{num}$ are almost identical (in fact they should be identical if the simulation results were fully exact and the system (\ref{eq:findimsphSteiner}) were numerically well-posed). One can also note that analytical approximation $p_{apx}$ is also very close to the simulated values (one can of course play with the regression parameters and the results would be a bit different; we selected $\lambda_{reg}=1$, $p_{reg}=2$, and $q_{reg}=2$). Finally, the fully analytical upper bound $p_{ub}^{(ldp)}$ is in our view pretty close to the simulated values as well given how daunting task we actually had in front of us at the beginning. We should emphasize once again that the finite dimensional situation considered here is pretty much the only scenario where the high-dimensional geometry has a bit of an edge compared to the probabilistic theories that we have developed and used to obtain $p_{ub}^{(ldp)}$. Also, as block-length grows, the gap is even smaller and as figures indicate the transition zone is shrinking (this is of course to be expected based even just on a mere fact that the system dimensions are larger in the right side portions of the figures; moreover, this is also the key reason why we selected the smallest block length, i.e. $d=2$, basically ensuring that the scenario where the gap is the largest is showcased as well).

%\begin{figure}[htb]
%%\begin{minipage}[b]{.5\linewidth}
%\centering
%\centerline{\epsfig{figure=crftl1bld2prerrallPAP.eps,width=11.5cm,height=8cm}}
%%\end{minipage}
%%\begin{minipage}[b]{.5\linewidth}
%%\centering
%%\centerline{\epsfig{figure=finprerral08.eps,width=9cm,height=6.5cm}}
%%\end{minipage}
%\caption{Probability or error estimates, $\ell_2/\ell_1$-optimization; $k=6$, $n=18$}
%\label{fig:findimd2crft}
%\end{figure}

\begin{figure}[htb]
\begin{minipage}[b]{.5\linewidth}
\centering
\centerline{\epsfig{figure=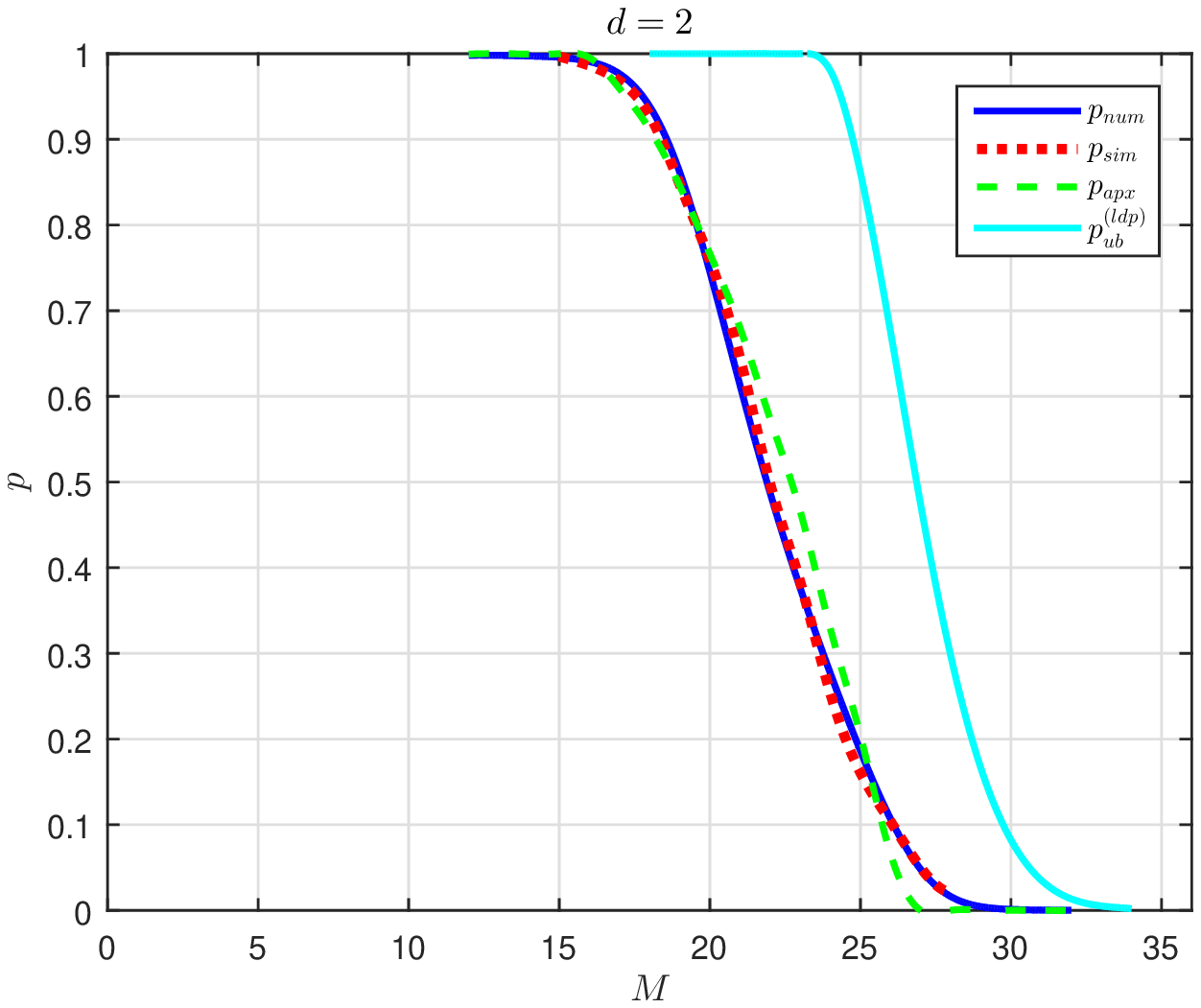,width=9cm,height=7cm}}
%\end{minipage}
%\begin{minipage}[b]{.5\linewidth}
%\centering
%\centerline{\epsfig{figure=finprerral08.eps,width=9cm,height=6.5cm}}
\end{minipage}
\begin{minipage}[b]{.5\linewidth}
\centering
\centerline{\epsfig{figure=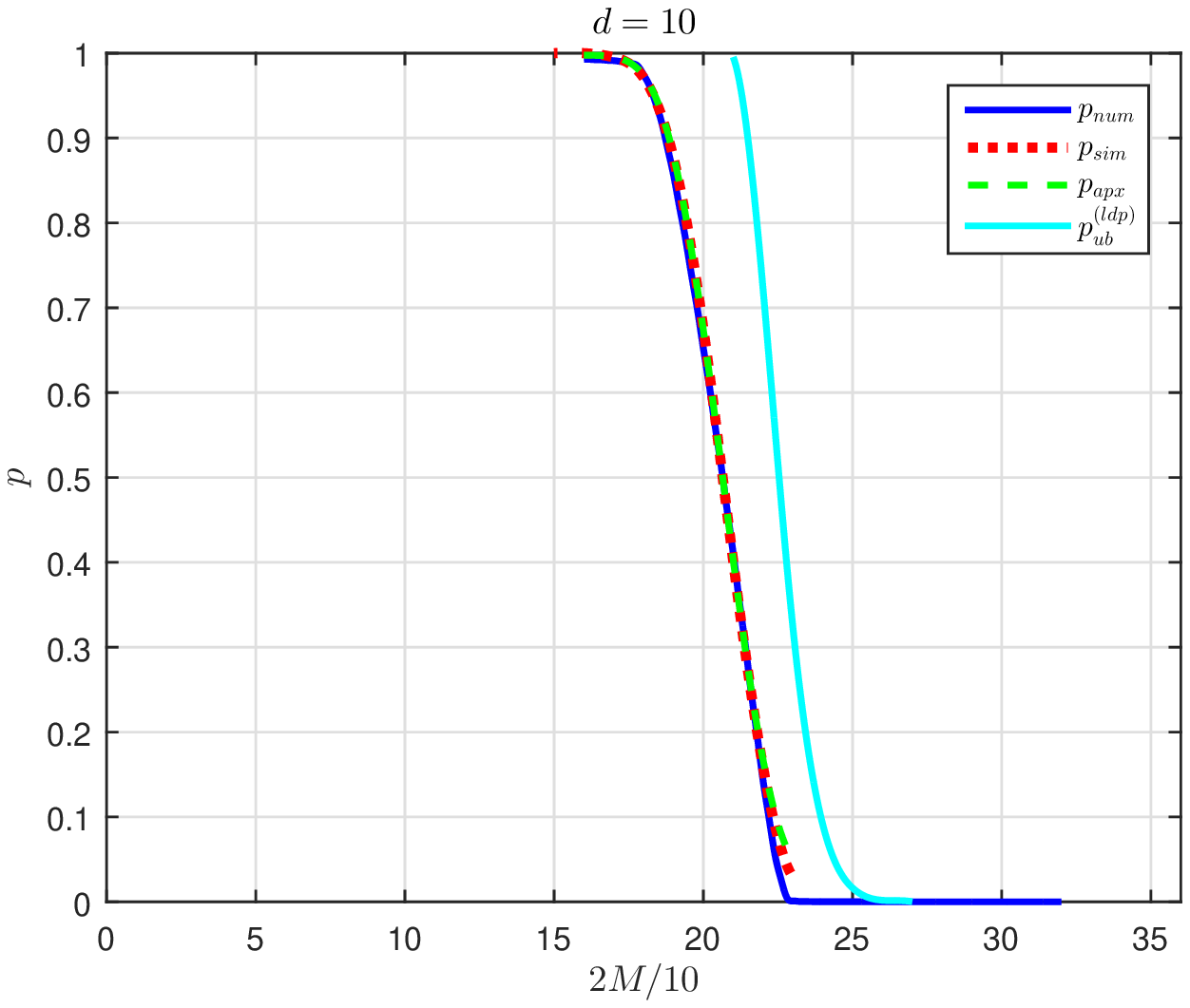,width=9cm,height=7cm}}
%\end{minipage}
%\begin{minipage}[b]{.5\linewidth}
%\centering
%\centerline{\epsfig{figure=finprerral08.eps,width=9cm,height=6.5cm}}
\end{minipage}
\caption{Probability of error estimates, $\ell_2/\ell_1$-optimization; $k=6$, $n=18$}
\label{fig:findimd2d10crft}
\end{figure}

%\begin{figure}[htb]
%%\begin{minipage}[b]{.5\linewidth}
%\centering
%\centerline{\epsfig{figure=crftl1bld10prerrallPAP.eps,width=11.5cm,height=8cm}}
%%\end{minipage}
%%\begin{minipage}[b]{.5\linewidth}
%%\centering
%%\centerline{\epsfig{figure=finprerral08.eps,width=9cm,height=6.5cm}}
%%\end{minipage}
%\caption{Probability or error estimates, $\ell_2/\ell_1$-optimization; $k=6$, $n=18$}
%\label{fig:findimd10crft}
%\end{figure}

\begin{figure}[htb]
\begin{minipage}[b]{.5\linewidth}
\centering
\centerline{\epsfig{figure=crftl1bld2prerrallPAP.eps,width=9cm,height=7cm}}
%\end{minipage}
%\begin{minipage}[b]{.5\linewidth}
%\centering
%\centerline{\epsfig{figure=finprerral08.eps,width=9cm,height=6.5cm}}
\end{minipage}
\begin{minipage}[b]{.5\linewidth}
\centering
\centerline{\epsfig{figure=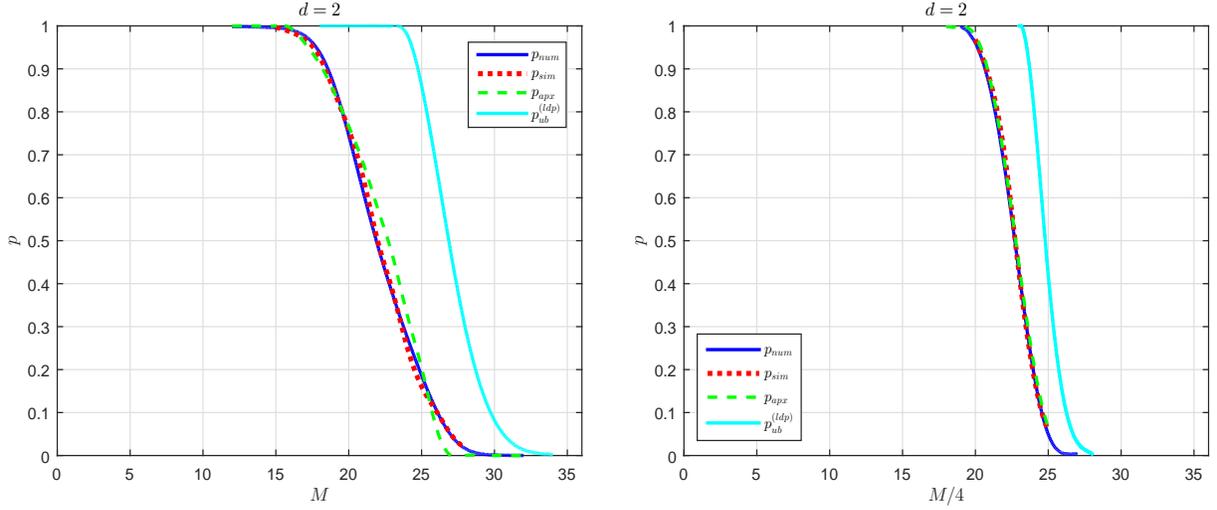,width=9cm,height=7cm}}
%\end{minipage}
%\begin{minipage}[b]{.5\linewidth}
%\centering
%\centerline{\epsfig{figure=finprerral08.eps,width=9cm,height=6.5cm}}
\end{minipage}
\caption{Probability of error estimates, $\ell_2/\ell_1$-optimization; left -- $k=6$, $n=18$; right -- $k=24$, $n=72$}
\label{fig:findimd2scaledcrft}
\end{figure}

%%%%%%%%%%%%%%%%%%%%%%%%%%%%%%%%%%%%%%%%%%%%%%%%%%%%%%%%%%%%%%%%%
\subsection{Probabilistic approach}
\label{sec:probapproach}
%%%%%%%%%%%%%%%%%%%%%%%%%%%%%%%%%%%%%%%%%%%%%%%%%%%%%%%%%%%%%%%%%

As we stated above, in Figures \ref{fig:findimd2d10crft} and \ref{fig:findimd2scaledcrft}, a fully analytical upper bound on the probability of error (denoted by $p_{ub}^{(ldp)}$ and obtained using results of Theorem \ref{thm:ldp1}) was plotted as well. In this subsection we will provide another way to upper-bound the probability of error. It will turn out that the results one can obtain through such an approach are similar to $p_{ub}^{(ldp)}$. However, the technique that we will present is of independent interest and we view it as a fairly general and useful tool.

We start things off by recalling on (\ref{eq:ldpprob})
\begin{equation}
P_{err}= P(\min_{\w\in S_w'}\|A\w\|_2\leq 0)=P(\max_{\w\in S_w'}\min_{\|\y\|_2=1}(\y^T A\w )\geq 0).
\label{eq:probapproachldpprob}
\end{equation}
Relying on the same trick that was utilized in (\ref{eq:ldpproblower}) we then have
\begin{equation}
P_{err}= P(\min_{\w\in S_w'}\|A\w\|_2\leq 0)=P(\max_{\w\in S_w'}\min_{\|\y\|_2=1}(\y^T A\w )\geq 0)\leq
P(\max_{\w\in S_w'}\min_{\|\y\|_2=1}\y^T A\w+(g-t_1)\geq 0)/P(g\geq t_1).
\label{eq:probapproachldpprob1}
\end{equation}
Adjusting Lemma \ref{lemma:unsignedlemmalower} (essentially its original form from \cite{Gordon85}) one then  has
\begin{eqnarray}
P_{err}&\leq&
P(\max_{\w\in S_w'}\min_{\|\y\|_2=1}\y^T A\w+(g-t_1)\geq 0)/P(g\geq t_1)\nonumber \\
&\leq&
P(\max_{\w\in S_w'}\min_{\|\y\|_2=1}(\y^T \g+\h^T\w-t_1)\geq 0)/P(g\geq t_1)\nonumber \\
&\leq&
P(-\|\g\|_2+w(\h,S_w')-t_1\geq 0)/P(g\geq t_1)\nonumber \\
&=&
P(-\|\g\|_2+w(\h,S_w)-t_1\geq 0)/P(g\geq t_1)\nonumber \\
&\leq&
\min_{\lambda\geq 0}P(-\|\g\|_2+w_{\lambda}(\h,S_w)-t_1\geq 0)/P(g\geq t_1).\nonumber \\
&\leq&
\min_{\lambda\geq 0}P_{\lambda,ub}/P(g\geq t_1),\nonumber \\
\label{eq:probapproachldpprob2}
\end{eqnarray}
where clearly we set
\begin{equation}\label{eq:probapproachldpprob3}
  P_{\lambda,ub} =  P(-\|\g\|_2+w_{\lambda}(\h,S_w)-t_1\geq 0).
\end{equation}
We then make use of the indicator step function similarly as in (\ref{eq:sigmast5})
\begin{eqnarray}\label{eq:probapproachsigmast5}
P_{\lambda,ub}& = & P(-\|\g\|_2+w_{\lambda}(\h,S_w)-t_1\geq 0)\nonumber \\
& = &
 \frac{1}{\sqrt{2\pi}^{M}} \frac{1}{\sqrt{2\pi}^{dn}}\int_{\g,\h}h(-\|\g\|_2+w_{\lambda}(\h,S_w)-t_1)e^{-\frac{\|\h\|^2+\|\g\|^2}{2}}d\h d\g\nonumber\\
& = &
 \frac{1}{\sqrt{2\pi}^{M}} \frac{1}{\sqrt{2\pi}^{dn}}\int_{\|\g\|_2\geq \max(-t_1,0),\h}h((w_{\lambda}(\h,S_w))^2-(t_1+\|\g\|_2)^2)e^{-\frac{\|\h\|^2+\|\g\|^2}{2}}d\h d\g\nonumber\\
 &&+
 \frac{1}{\sqrt{2\pi}^{M}}\int_{\|\g\|_2\leq \max(-t_1,0)}e^{-\frac{\|\g\|^2}{2}}d\g\nonumber \\
& = &
  P_{\lambda,ub}^{(1)} +\left (\int_{0\leq z_g\leq (\max(-t_1,0))^2}
   \frac{2^{-M/2}}{\Gamma(M/2)}z_g^{M/2-1}e^{-z_g/2}dz_g\right ),
\end{eqnarray}
where
\begin{equation}\label{eq:probapproachsigmast51}
  P_{\lambda,ub}^{(1)}=\frac{1}{\sqrt{2\pi}^{M}} \frac{1}{\sqrt{2\pi}^{dn}}\int_{\|\g\|_2\geq \max(-t_1,0),\h}h((w_{\lambda}(\h,S_w))^2-(t_1+\|\g\|_2)^2)e^{-\frac{\|\h\|^2+\|\g\|^2}{2}}d\h d\g.
\end{equation}
Combining (\ref{eq:sigmasthfun}), (\ref{eq:probapproachsigmast5}), and (\ref{eq:probapproachsigmast51}) we then have
\begin{eqnarray}\label{eq:probapproachsigmast6}
P_{\lambda,ub}^{(1)} &=&  \frac{1}{\sqrt{2\pi}^{M}} \frac{1}{\sqrt{2\pi}^{dn}}\int_{\|\g\|_2\geq \max(-t_1,0),\h}h((w_{\lambda}(\h,S_w))^2-(t_1+\|\g\|_2)^2))e^{-\frac{\|\h\|_2^2+\|\g\|^2}{2}}d\h d\g\nonumber\\
   &=& \frac{1}{\sqrt{2\pi}^{M}} \frac{1}{\sqrt{2\pi}^{dn}}\int_{\|\g\|_2\geq \max(-t_1,0),\h}
   \lim_{\epsilon_{wf}\rightarrow 0_+}\int_{-\infty}^{\infty}\frac{e^{jw_f((w_\lambda(\h,S_w))^2-(t_1+\|\g\|_2)^2)}}{2\pi j(w_f-j\epsilon_{wf})}dw_f
   e^{-\frac{\|\h\|_2^2+\|\g\|_2^2}{2}}d\h d\g \nonumber\\
   &=&
   \lim_{\epsilon_{wf}\rightarrow 0_+}\int_{-\infty}^{\infty}\frac{1}{{2\pi j(w_f-j\epsilon_{wf})}}\int_{\|\g\|_2\geq \max(-t_1,0),\h}e^{jw_f(w_{\lambda,1}+w_{\lambda,2}+w_{\lambda,3}-(t_1+\|\g\|_2)^2)}
   \frac{e^{-\frac{\|\h\|_2^2+\|\g\|_2^2}{2}}}{\sqrt{2\pi}^{M+dn}}d\h d\g dw_f \nonumber\\
   &=&
   \lim_{\epsilon_{wf}\rightarrow 0_+}\int_{-\infty}^{\infty}\frac{1}{{2\pi j(w_f-j\epsilon_{wf})}}I_{\lambda,1}(-w_f,\pi/2)I_{\lambda,2}(-w_f,\pi/2)I_{\lambda,3}(-w_f,\pi/2)I_{\lambda,4}(w_f,t_1)dw_f,
\end{eqnarray}
where
\begin{eqnarray}\label{eq:probapproachsigmast7}
  I_{\lambda,4}(w_f,t_1) &=& \frac{1}{\sqrt{2\pi}^{M}}\int_{\|\g\|_2\geq \max(-t_1,0)} e^{-jw_f(t_1+\|\g\|_2)^2}
   e^{-\frac{\|\g\|_2^2}{2}}d\g\nonumber\\
   &=&\left (\int_{z_g\geq (\max(-t_1,0))^2}e^{-jw_f(t_1+\sqrt{z_g})^2}
   \frac{2^{-M/2}}{\Gamma(M/2)}z_g^{M/2-1}e^{-z_g/2}d\z_g\right ).
\end{eqnarray}
We summarize the above discussion in the following theorem.
\begin{theorem}
Assume the setup of Theorem \ref{thm:ldp1}. Accordingly, let $P_{err}$ be the probability that the solution of (\ref{eq:l2l1}) is not the $k$-block-sparse solution of (\ref{eq:system}). Then
\begin{eqnarray}
P_{err}&\leq &   \min_{t_1}\min_{\lambda\geq 0} (\lim_{\epsilon_{wf}\rightarrow 0_+}\int_{-\infty}^{\infty}\frac{1}{{2\pi j(w_f-j\epsilon_{wf})}}I_{\lambda,1}(-w_f,\pi/2)I_{\lambda,2}(-w_f,\pi/2)I_{\lambda,3}(-w_f,\pi/2)I_{\lambda,4}(w_f,t_1)dw_f\nonumber\\
&&+(\int_{0\leq z_g\leq (\max(-t_1,0))^2}
   \frac{2^{-M/2}}{\Gamma(M/2)}z_g^{M/2-1}e^{-z_g/2}dz_g ))/P(g\geq t_1)),
\label{eq:probapproachldpthm1perrub1}
\end{eqnarray}
where $I_{\lambda,1}(\cdot,\cdot)$, $I_{\lambda,2}(\cdot,\cdot)$, $I_{\lambda,3}(\cdot,\cdot)$, and $I_{\lambda,4}(\cdot,\cdot)$ are as given in (\ref{eq:sigmastIwf10}), (\ref{eq:sigmastIwf2}), (\ref{eq:sigmastIwf30}), and (\ref{eq:probapproachsigmast7}) respectively.
\label{thm:probapproachldp1}
\end{theorem}
\begin{proof}
  Follows from the previous discussion.
\end{proof}
In Figure \ref{fig:findimd2addgub} we present the results that one can obtain based on the above theorem. Quantity denoted by $p_{ub}^{(ag)}$ is the upper bound on $P_{err}$ that we computed utilizing Theorem \ref{thm:probapproachldp1}. As can be seen (and as we hinted at the beginning of this subsection) the results are comparable to $p_{ub}^{(ldp)}$ (as expected this is even more emphasized as the system dimensions grow). However, the strategy that we developed above is a very powerful tool and we wanted to present it on its own without too much emphasis on the concrete values that it provides for the upper estimate of $P_{err}$.

%\begin{figure}[htb]
%%\begin{minipage}[b]{.5\linewidth}
%\centering
%\centerline{\epsfig{figure=addgubbettl1bld2prerrallPAP.eps,width=11.5cm,height=8cm}}
%%\end{minipage}
%%\begin{minipage}[b]{.5\linewidth}
%%\centering
%%\centerline{\epsfig{figure=finprerral08.eps,width=9cm,height=6.5cm}}
%%\end{minipage}
%\caption{Probability of error estimates, $\ell_2/\ell_1$-optimization; $k=6$, $n=18$}
%\label{fig:findimd2addgub}
%\end{figure}

\begin{figure}[htb]
\begin{minipage}[b]{.5\linewidth}
\centering
\centerline{\epsfig{figure=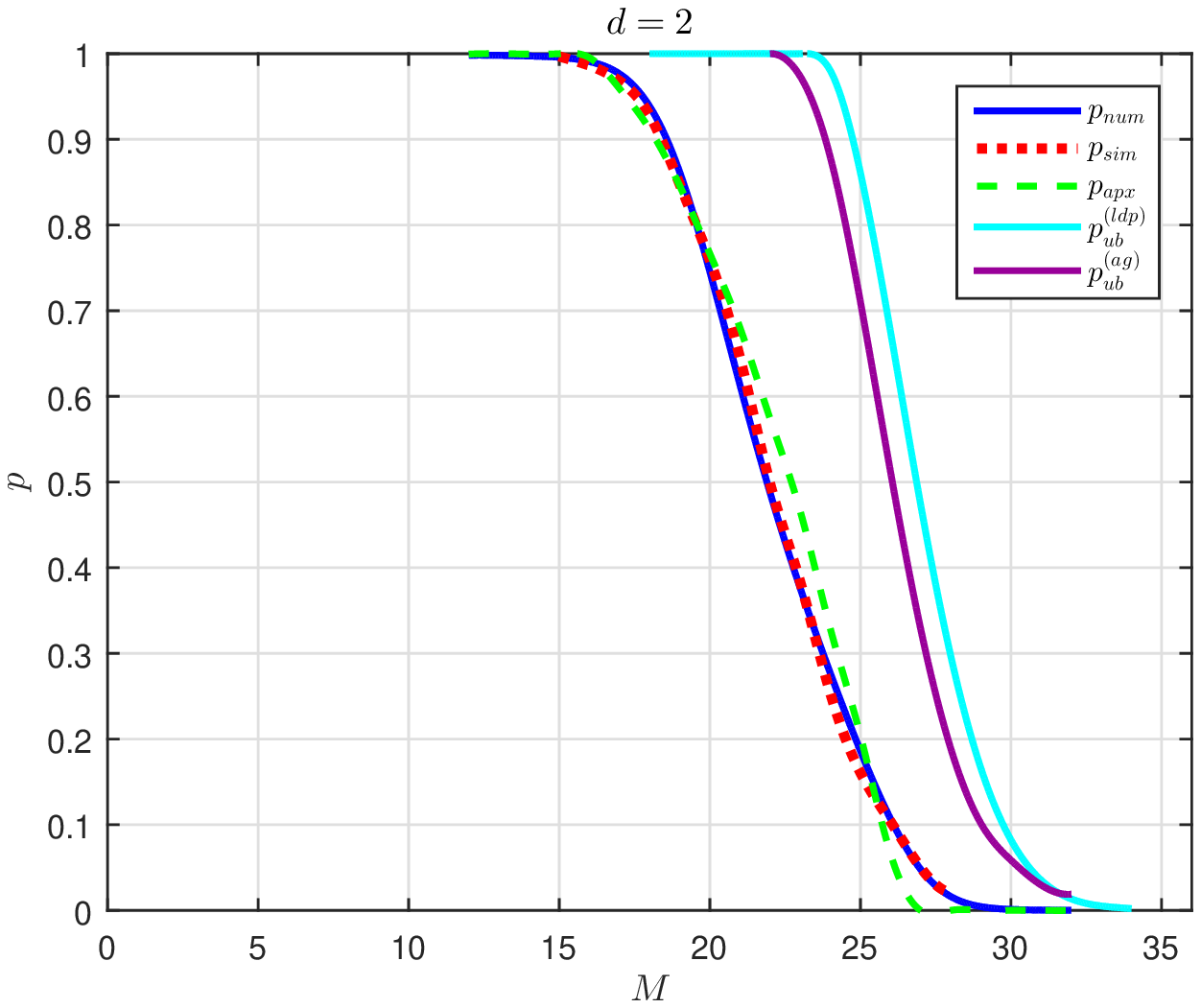,width=9cm,height=7cm}}
%\end{minipage}
%\begin{minipage}[b]{.5\linewidth}
%\centering
%\centerline{\epsfig{figure=finprerral08.eps,width=9cm,height=6.5cm}}
\end{minipage}
\begin{minipage}[b]{.5\linewidth}
\centering
\centerline{\epsfig{figure=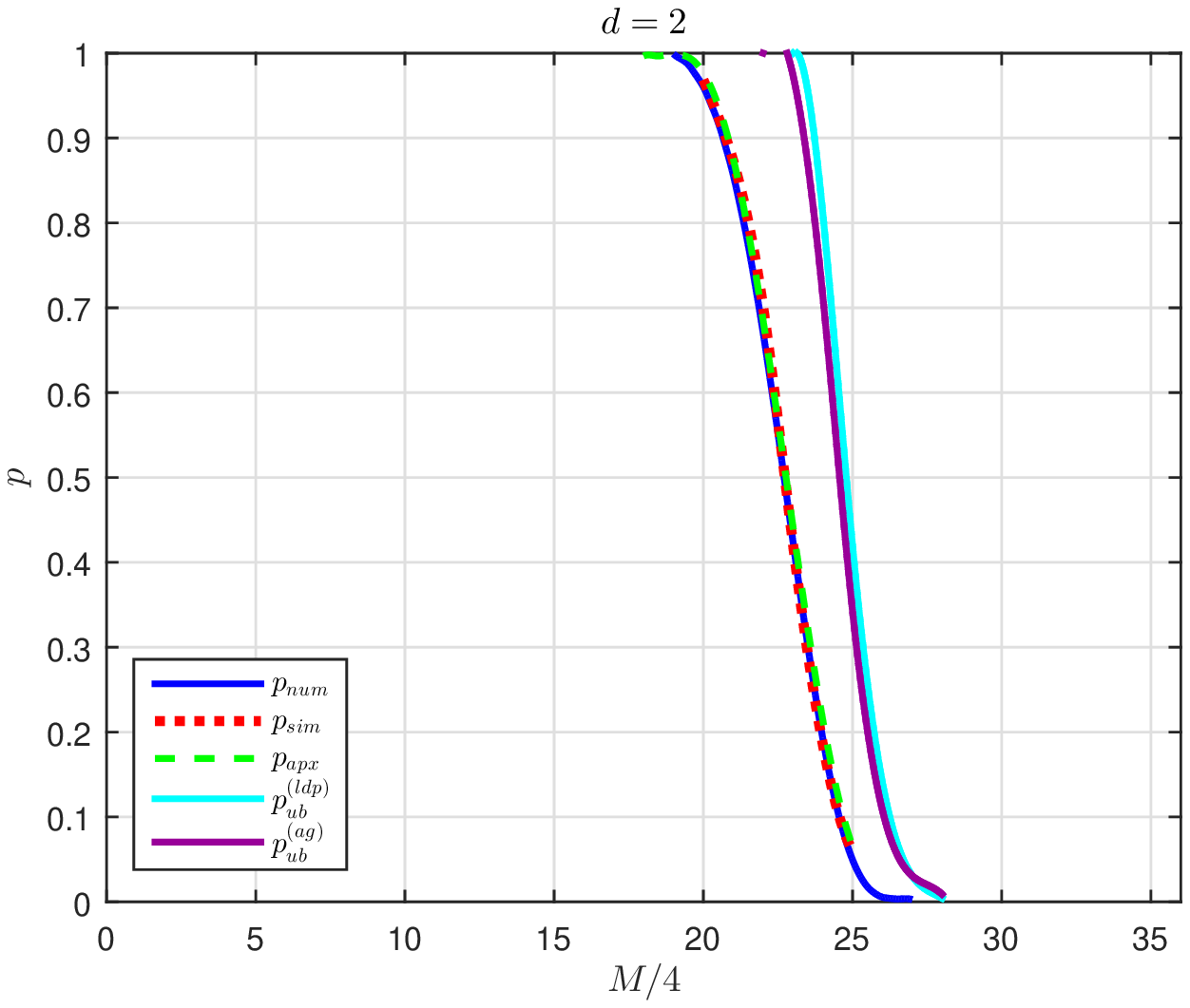,width=9cm,height=7cm}}
%\end{minipage}
%\begin{minipage}[b]{.5\linewidth}
%\centering
%\centerline{\epsfig{figure=finprerral08.eps,width=9cm,height=6.5cm}}
\end{minipage}
\caption{Probability of error estimates, $\ell_2/\ell_1$-optimization; left -- $k=6$, $n=18$; right -- $k=24$, $n=72$}
\label{fig:findimd2addgub}
\end{figure}

\section{Conclusion}
\label{sec:conc}
%%%%%%%%%%%%%%%%%%%%%%%%%%%%%%%%%%%%%%%%%%%%%%%%%%%%%%%%%%%%%%%%%%%%%%%%%%%%%%%%

The main topic of interest in this paper are random under-determined systems of linear equations with block-sparse solutions. A particular line of study that focuses on a special class of the so called $\ell_2/\ell_1$-optimization algorithms was considered. We first revisited a few earlier results that relate to an asymptotic analysis of such algorithms when employed for solving random linear systems with block-sparse solutions. In particular, we first revisited the so-called phase-transition phenomena. These phenomena provide the exact characterization of the so-called breaking points, i.e. of the points where the algorithms sharply transition from being perfectly reliable to being unable to solve problems at hand. In a series of earlier works we introduced a systematic way for studying such a phenomena that relies on a purely probabilistic approach.

We then went a bit further and created a completely new concept that we in a way connected to the large deviation theory. Namely, we introduced a large deviation type of considerations for the probabilities of error which in essence deal with the tail behavior of the phase-transition curves. These are again by their nature asymptotic and we were able to show that the mechanisms we designed earlier can in fact be restructured so that they can handle such a type of problems as well.

In addition to the asymptotic considerations, we also raised the level a bit and considered finite dimensional scenarios as well (of course these being the ultimate considerations one would like to have). We then showed that the probabilistic theories that we have built can in fact produce a solid finite dimensional estimates as well.

Finally, we also considered a high-dimensional integral geometry approach as an alternative tool for the types of the analyses of interest in this paper. We introduced several novel technical steps that enable such an approach to work which eventually enabled us to obtain excellent estimates for success/failure of the $\ell_2/\ell_1$-optimization when used for solving linear systems with block-sparse solutions.

Of course, as is often the case with the type of results that we showcased here, one can then continue further and consider various other aspects/extensions of the algorithms/problems at hand and adjust the techniques introduced here and in a few of our earlier works so that they fit those problems as well. Such adjustments are fairly standard and for a few particularly interesting problems we will in several companion papers present how one can do them and what kind of results one can eventually obtain through them.

%\newpage1
%\setcounter{page}{1}
\begin{singlespace}
\bibliographystyle{plain}
\bibliography{blspfinnldpasym1Refs}

\begin{thebibliography}{10}

\bibitem{BCDH08}
R.~Baraniuk, V.~Cevher, M.~Duarte, and C.~Hegde.
\newblock Model-based compressive sensing.
\newblock available online at \bl{\url{http://www.dsp.ece.rice.edu/cs/}}.

\bibitem{BWDSB05}
D.~Baron, M.~Wakin, M.~Duarte, S.~Sarvotham, and Richard Baraniuk.
\newblock Distributed compressed sensing.
\newblock {\em Allerton}, 2005.

\bibitem{BluDav09}
T.~Blumensath and M.~E. Davies.
\newblock Sampling theorems for signals from the union of finite-dimensional
  linear subspaces.
\newblock {\em IEEE Transactions on Information Theory}, 55(4):187--1882, 2009.

\bibitem{CRT}
E.~Candes, J.~Romberg, and T.~Tao.
\newblock Robust uncertainty principles: exact signal reconstruction from
  highly incomplete frequency information.
\newblock {\em IEEE Trans. on Information Theory}, 52(12):489--509, 2006.

\bibitem{CeInHeBa09}
V.~Cevher, P.~Indyk, C.~Hegde, and R.~G. Baraniuk.
\newblock Recovery of clustered sparse signals from compressive measurements.
\newblock {\em SAMPTA, International Conference on Sampling Theory and
  Applications}, 2009.
\newblock Marseille, France.

\bibitem{CH06}
J.~Chen and X.~Huo.
\newblock Theoretical results on sparse representations of multiple-measurement
  vectors.
\newblock {\em IEEE Trans. on Signal Processing}, 54(12):4634--4643, 2006.

\bibitem{CREKD}
S.~Cotter, B.~Rao, K.~Engan, and K.~Kreutz-Delgado.
\newblock Sparse solutions to linear inverse problems with multiple measurement
  vectors.
\newblock {\em IEEE Trans. on Signal Porcessing}, 53(7):2477 -- 2488, 2005.

\bibitem{DaiMil08}
W.~Dai and O.~Milenkovic.
\newblock Subspace pursuit for compressive sensing signal reconstruction.
\newblock available online at \bl{\url{http://arxiv.org/abs/0803.0811}}.

\bibitem{DonohoUnsigned}
D.~Donoho.
\newblock Neighborly polytopes and sparse solutions of underdetermined linear
  equations.
\newblock 2004.
\newblock Technical report, Department of Statistics, Stanford University.

\bibitem{DonohoPol}
D.~Donoho.
\newblock High-dimensional centrally symmetric polytopes with neighborlines
  proportional to dimension.
\newblock {\em Disc. Comput. Geometry}, 35(4):617--652, 2006.

\bibitem{DonMalMon09}
D.~Donoho, A.~Maleki, and A.~Montanari.
\newblock Message-passing algorithms for compressed sensing.
\newblock {\em Proc. National Academy of Sciences}, 106(45):18914--18919, 2009.

\bibitem{DOnoho06CS}
D.~L. Donoho.
\newblock Compressed sensing.
\newblock {\em IEEE Trans. on Information Theory}, 52(4):1289--1306, 2006.

\bibitem{DTDSomp}
D.~L. Donoho, Y.~Tsaig, I.~Drori, and J.L. Starck.
\newblock Sparse solution of underdetermined linear equations by stagewise
  orthogonal matching pursuit.
\newblock {\em 2007}.
\newblock available online at \bl{\url{http://www.dsp.ece.rice.edu/cs/}}.

\bibitem{EldBol09}
Y.~C. Eldar and H.~Bolcskei.
\newblock Block-sparsity: Coherence and efficient recovery.
\newblock {\em ICASSP, IEEE International Conference on Acoustics, Signal and
  Speech Processing}, pages 2885--2888, April 2009.
\newblock Taipei, Taiwan.

\bibitem{EKB09}
Y.~C. Eldar, P.~Kuppinger, and H.~Bolcskei.
\newblock Compressed sensing of block-sparse signals: Uncertainty relations and
  efficient recovery.
\newblock {\em IEEE Trans. on Signal Processing}, 58(6):3042--3054, 2010.

\bibitem{EMsub}
Y.~C. Eldar and M.~Mishali.
\newblock Robust recovery of signals from a structured union of subspaces.
\newblock {\em IEEE Trans. on Information Theory}, 55(1):5302--5316, 2009.

\bibitem{EldRau09}
Y.~C. Eldar and H.~Rauhut.
\newblock Average case analysis of multichannel sparse recovery using convex
  relaxation.
\newblock {\em IEEE Trans. on Information Theory}, 56(1):505--519, 2010.

\bibitem{GaZhMa09}
A.~Ganesh, Z.~Zhou, and Y.~Ma.
\newblock Separation of a subspace-sparse signal: Algorithms and conditions.
\newblock {\em ICASSP, IEEE International Conference on Acoustics, Speech and
  Signal Processing}, pages 3141--3144, April 2009.
\newblock Taipei, Taiwan.

\bibitem{Gordon85}
Y.~Gordon.
\newblock Some inequalities for gaussian processes and applications.
\newblock {\em Israel Journal of Mathematics}, 50(4):265--289, 1985.

\bibitem{MEldar}
M.~Mishali and Y.~Eldar.
\newblock Reduce and boost: Recovering arbitrary sets of jointly sparse
  vectors.
\newblock {\em IEEE Trans. on Signal Processing}, 56(10):4692--4702, 2008.

\bibitem{NT08}
D.~Needell and J.~A. Tropp.
\newblock {CoSaMP}: Iterative signal recovery from incomplete and inaccurate
  samples.
\newblock {\em Applied and Computational Harmonic Analysis}, 26(3):301--321,
  2009.

\bibitem{NeVe07}
D.~Needell and R.~Vershynin.
\newblock Unifrom uncertainly principles and signal recovery via regularized
  orthogonal matching pursuit.
\newblock {\em Foundations of Computational Mathematics}, 9(3):317--334, 2009.

\bibitem{NegWai09}
S.~Negahban and M.~J. Wainwright.
\newblock Simultaneous support recovery in high dimensions: Benefits and perils
  of block $\ell_1/\ell_\infty$-regularization.
\newblock {\em IEEE Trans. on Information Theory}, 57(6):3841--3863, 2011.

\bibitem{FHicassp}
F.~Parvaresh and B.~Hassibi.
\newblock Explicit measurements with almost optimal thresholds for compressed
  sensing.
\newblock {\em ICASSP, IEEE International Conference on Acoustics, Signal and
  Speech Processing}, Mar-Apr 2008.
\newblock Las Vegas, NV.

\bibitem{SW08}
R.~Schneider and W.~Weil.
\newblock {\em High-dimensional integral geometry}.
\newblock Springer, 2008.

\bibitem{StojnicCSetamBlock09}
M.~Stojnic.
\newblock Block-length dependent thresholds in block-sparse compressed sensing.
\newblock available online at \bl{\url{http://arxiv.org/abs/0907.3679}}.

\bibitem{StojnicBlockDepNon10}
M.~Stojnic.
\newblock Compressed sensing of block-sparse positive vectors.
\newblock available online at \bl{\url{http://arxiv.org/abs/1306.3977}}.

\bibitem{StojnicMoreSophHopBnds10}
M.~Stojnic.
\newblock Lifting/lowering {H}opfield models ground state energies.
\newblock available online at \bl{\url{http://arxiv.org/abs/1306.3975}}.

\bibitem{StojnicReDirChall13}
M.~Stojnic.
\newblock Linear under-determined systems with sparse solutions: Redirecting a
  challenge?
\newblock available online at \bl{\url{http://arxiv.org/abs/1304.0004}}.

\bibitem{StojnicUpperBlock10}
M.~Stojnic.
\newblock Optimality of $\ell_2/\ell_1$-optimization block-length dependent
  thresholds.
\newblock available online at \bl{\url{http://arxiv.org/abs/1304.0001}}.

\bibitem{StojnicUpper10}
M.~Stojnic.
\newblock Upper-bounding $\ell_1$-optimization weak thresholds.
\newblock available online at \bl{\url{http://arxiv.org/abs/1303.7289}}.

\bibitem{StojnicCSetam09}
M.~Stojnic.
\newblock Various thresholds for $\ell_1$-optimization in compressed sensing.
\newblock available online at \bl{\url{http://arxiv.org/abs/0907.3666}}.

\bibitem{StojnicICASSP09}
M.~Stojnic.
\newblock A simple performance analysis of $\ell_1$-optimization in compressed
  sensing.
\newblock {\em ICASSP, IEEE International Conference on Acoustics, Signal and
  Speech Processing}, pages 3021--3024, April 2009.
\newblock Taipei, Taiwan.

\bibitem{StojnicICASSP09block}
M.~Stojnic.
\newblock Strong thresholds for $\ell_2/\ell_1$-optimization in block-sparse
  compressed sensing.
\newblock {\em ICASSP, IEEE International Conference on Acoustics, Signal and
  Speech Processing}, pages 3025--3028, April 2009.
\newblock Taipei, Taiwan.

\bibitem{StojnicJSTSP09}
M.~Stojnic.
\newblock $\ell_2/\ell_1$-optimization in block-sparse compressed sensing and
  its strong thresholds.
\newblock {\em IEEE Journal of Selected Topics in Signal Processing},
  4(2):350--357, 2010.

\bibitem{StojnicISIT2010binary}
M.~Stojnic.
\newblock Recovery thresholds for $\ell_1$ optimization in binary compressed
  sensing.
\newblock {\em ISIT, IEEE International Symposium on Information Theory}, pages
  1593 -- 1597, 13-18 June 2010.
\newblock Austin, TX.

\bibitem{StojnicICASSP10knownsupp}
M.~Stojnic.
\newblock Towards improving $\ell_1$ optimization in compressed sensing.
\newblock {\em ICASSP, IEEE International Conference on Acoustics, Signal and
  Speech Processing}, pages 3938--3941, 14-19 March 2010.
\newblock Dallas, TX.

\bibitem{SPH}
M.~Stojnic, F.~Parvaresh, and B.~Hassibi.
\newblock On the reconstruction of block-sparse signals with an optimal number
  of measurements.
\newblock {\em IEEE Trans. on Signal Processing}, 57(8):3075--3085, 2009.

\bibitem{Temlyakov04}
V.N. Temlyakov.
\newblock A remark on simultaneous greedy approximation.
\newblock {\em East J. Approx.}, 100, 2004.

\bibitem{JATGomp}
J.~Tropp and A.~Gilbert.
\newblock Signal recovery from random measurements via orthogonal matching
  pursuit.
\newblock {\em IEEE Trans. on Information Theory}, 53(12):4655--4666, 2007.

\bibitem{TGS05}
J.~Tropp, A.~C. Gilbert, and M.~Strauss.
\newblock Algorithms for simultaneous sparse approximation. part i: Greedy
  pursuit.
\newblock {\em Signal Processing}, 86(3):572--588, 2006.

\bibitem{JAT}
J.~A. Tropp.
\newblock Greed is good: algorithmic results for sparse approximations.
\newblock {\em IEEE Trans. on Information Theory}, 50(10):2231--2242, 2004.

\bibitem{BerFri09}
E.~van~den Berg and M.~P. Friedlander.
\newblock Joint-sparse recovery from multiple measurements.
\newblock available online at \bl{\url{https://arxiv.org/abs/0904.2051}}.

\bibitem{VPH}
H.~Vikalo, F.~Parvaresh, and B.~Hassibi.
\newblock On sparse recovery of compressed dna microarrays.
\newblock {\em Asilomor conference}, November 2007.

\bibitem{ZeGoAd09}
A.~C. Zelinski, V.~K. Goyal, and E.~Adalsteinsson.
\newblock Simultaneously sparse solutions to linear inverse problems with
  multiple system matrices and a single observation vector.
\newblock {\em SIAM J. Sci. Comput.}, 31(6):4533--4579, 2010.

\bibitem{ZeWaSeGoAd08}
A.~C. Zelinski, L.~L. Wald, K.~Setsompop, V.~K. Goyal, and E.~Adalsteinsson.
\newblock Sparsity-enforced slice-selective mri rf excitation pulse design.
\newblock {\em IEEE Trans. on Medical Imaging}, 27(9):1213--1229, 2008.

\end{thebibliography}
\end{singlespace}

\end{document}